\documentclass[10pt]{amsart}
\usepackage{tikz-cd}
\usepackage{amssymb}
\usepackage{graphicx}
\usepackage{yhmath}
\usepackage{extpfeil}
\usepackage{epsfig,graphics}
\usepackage[all,knot,poly,color]{xy}
\usepackage{multicol}
\usepackage{array}
\usepackage{makecell}
\usepackage{enumitem}
\usepackage{fullpage}
\usepackage{tikz-cd}
\usepackage{amssymb}
\usepackage{amsmath}
\usepackage{amsxtra}
\usepackage{amscd}
\usepackage{graphicx}
\usepackage{amsfonts}
\usepackage{pb-diagram}
\usepackage{float}
\usepackage{enumerate}
\usepackage{dsfont}
\usepackage{multirow}
\usepackage{color}
\usepackage{rotating}
\usepackage{url}
\usepackage{enumitem}
\usepackage[bookmarks=false]{hyperref}
\definecolor{darkgreen}{rgb}{0,0.4,0.1}
\definecolor{darkpurple}{rgb}{0.7,0.0,0.4}
\hypersetup{
   colorlinks=true,       	% false: boxed links; true: colored links
   linkcolor=darkpurple,     % color of internal links 
   citecolor=darkgreen,    % color of links to bibliography
   filecolor=magenta,      % color of file links
   urlcolor=blue			% color of external links
}

\numberwithin{equation}{section}
\newtheorem{thm}{Theorem}[section]
\newtheorem{lem}[thm]{Lemma}
\newtheorem{cor}[thm]{Corollary}
\newtheorem{prop}[thm]{Proposition}

\newtheorem{conj}[thm]{Conjecture}
\newtheorem{exmp}[thm]{Example}

\newtheorem{rem}[thm]{Remark}

\def\N{{\mathbb N}}
\def\Z{{\mathbb Z}}

\def\C{{\mathbb C}}

\def\ZZ{\mathbb{Z}}
\def\NN{\mathbb{N}}

\newcommand{\gra}{{\alpha}} \newcommand{\grb}{{\beta}} \newcommand{\grg}{{\gamma}} 
   \newcommand{\gru}{{\theta}}

\newcommand{\grf}{{\phi}}  \newcommand{\grc}{{\psi}}

  \newcommand{\grC}{{\Psi}} 
\providecommand{\udots}{\Udots}
\makeatletter
\DeclareRobustCommand{\Udots}{%
	\vcenter{\offinterlineskip
		\halign{%
			\hbox to .9em{##}\cr
			\hfil.\cr\noalign{\kern.4ex}
			\hfil.\hfil\cr\noalign{\kern.4ex}
			.\hfil\cr}%
	}%
}
\makeatother

\title{The freeness and trace conjectures\\ for parabolic Hecke subalgebras}

\author{Eirini Chavli}
\address{Institut f\"ur Algebra und Zahlentheorie, Universit\"at Stuttgart, Pfaffenwaldring 57, 
70569 Stuttgart, Germany.}
\email{eirini.chavli@mathematik.uni-stuttgart.de}

\author{Maria Chlouveraki}
\address{Universit\'e Paris-Saclay, UVSQ, CNRS, Laboratoire de Math\'ematiques de Versailles, 
45 avenue des Etats-Unis,
78000 Versailles, France}
\email{maria.chlouveraki@uvsq.fr}


\begin{document}
	%!TeX spellcheck = en_GB 
\thanks{The second author is supported by the \emph{Agence Nationale de la Recherche} through the JCJC project ANR-18-CE40-0001.}

\maketitle

\begin{abstract}
The two most fundamental conjectures on the structure of the generic Hecke algebra $\mathcal{H}(W)$ associated with a complex reflection group $W$ state that $\mathcal{H}(W)$ is a free module of rank $|W|$ over its ring of definition, and that $\mathcal{H}(W)$ admits a canonical symmetrising trace. The first conjecture has recently become a theorem, while the second conjecture, known to hold for real reflection groups,  has only been proved for some exceptional non-real complex reflection groups (all of rank $2$ but one).
The two most fundamental conjectures on the structure of the parabolic Hecke subalgebra $\mathcal{H}(W')$ associated with a parabolic subgroup $W'$ of $W$ state that $\mathcal{H}(W)$ is a free left and right $\mathcal{H}(W')$-module of rank $|W|/|W'|$, and that the canonical symmetrising trace of $\mathcal{H}(W')$ is the restriction of the  canonical symmetrising trace of $\mathcal{H}(W)$ to $\mathcal{H}(W')$. Until now, these two conjectures have only be known to be true for real reflection groups. We prove them for all complex reflection groups of rank $2$ for which the BMM symmetrising trace conjecture is known to hold. 
\end{abstract}

\section{Introduction}\label{Introduction}

Real reflection groups are finite groups of real matrices generated by reflections. They include the Weyl groups, which are fundamental in the classification and  study of other algebraic structures, such as finite reductive groups and complex Lie algebras. A finite group $W$ is a real reflection group if and only if it is a finite Coxeter group, that is, it has a presentation of the form
$$W=\langle s\in S\;|\; (st)^{m_{st}}=1
\text{ for all } s,t \in S\rangle$$
where $m_{st}\in \mathbb{Z}_{\geq2}$ for $s\not=t$ and $m_{ss}=1$.  
The subgroups of $W$ generated by subsets of $S$ are called (standard) parabolic subgroups. These are very special subgroups, which share many defining properties of the parent group $W$. They are  real reflection groups on their own right, and their cosets can be represented by canonical elements of $W$. They also allow us to study the parent group $W$ using induction arguments, including for the determination of its conjugacy classes and the calculation of its character table (see \cite[Chapters 2, 3 and 6]{gepf}).  Finally, Lusztig's families of characters \cite{26}, which play a key role in the representation theory of finite reductive groups, are defined through a truncated induction from parabolic subgroups (Lusztig's definition uses generic degrees and the theory of Iwahori--Hecke algebras; for a different approach see also \cite[Chapter 6]{gepf}).

Complex reflection groups  are finite groups of complex matrices generated by pseudo-reflections, that is, non-trivial elements that fix a hyperplane pointwise; they thus include and generalise real reflection groups. A complex reflection group is isomorphic to a product of irreducible complex reflection groups, which are classified as follows: they either belong to the infintie series $G(de,e,n)$, where $d,e,n \in \N^*$, or they are one of the $34$ exceptional groups $G_4,G_5,\ldots,G_{37}$. 
If $W \subset GL(V)$ is a complex reflection group, where $V$ is a finite-dimensional complex vector space, then 
the parabolic subgroups of $W$ are defined to be the pointwise stabilisers of the subspaces of $V$ --- in the real case, with this definition, the parabolic subgroups are all conjugates of standard parabolic subgroups. All parabolic subgroups of $W$ are also complex reflection groups. In particular, if $V$ is of dimension $2$ (the dimension of $V$ is also called the rank of $W$, if $W$ is irreducible), then all non-trivial proper parabolic subgroups of $W$ are cyclic groups.

Iwahori--Hecke algebras associated with Weyl groups appear as endomorphism algebras of induced representations in the study of finite reductive groups. They can also be defined independently as deformations of the group algebras of the associated Weyl groups, and this definition can be applied to all real reflection groups. Equivalently, they can be  seen as quotients of the group algebras of the corresponding braid groups. The ring of definition  of a generic Iwahori--Hecke algebra is a Laurent polynomial ring. Specialising the parameters of this ring accordingly, one can view the generic Hecke algebras of parabolic subgroups as subalgebras of the Iwahori--Hecke algebra of the parent group (this is not the case for any reflection subgroup). 
Similarly to the finite group case, we can use induction arguments to determine the representations of the latter from the representations of these parabolic Hecke subalgebras. 

Inspired by the idea that some complex reflection groups could play the role of Weyl groups of objects generalising finite reductive groups, the so-called ``Spetses'' (see \cite{BMM}), Brou\'e, Malle and Rouquier generalised the notion of braid groups and Hecke algebras to the case of complex reflection groups in \cite{BMR}. Since then, the study of these objects has grown to a subject on its own right, and Hecke algebras associated with complex reflection groups have turned out to be connected to many other mathematical structures, from other algebras, such as Cherednik algebras and quantum groups, to other theories, such as knot theory and mathematical physics. However, for everything to make sense and work as in the real case, two fundamental conjectures, which were standard facts in the real case, had to be stated. 
We will also state them here, using the notation $\mathcal{H}(W)$ for the generic Hecke algebra associated with a complex reflection group $W$ and $R(W)$ for the Laurent polynomial ring over which $\mathcal{H}(W)$ is defined.

$ $\\
\textbf{The BMR freeness conjecture} \cite{BMR} \emph{The algebra $\mathcal{H}(W)$ is a free $R(W)$-module of rank $|W|$.}\\

This conjecture has been a theorem for the past couple of years. It has been tackled with a case-by-case analysis, using combinatorial methods for the groups of the infinite series, and computational methods for the exceptional groups. For most of the exceptional groups of rank $2$, the conjecture has been proved by the first author in \cite{Ch18,Ch17}. 

$ $\\
\textbf{The BMM symmetrising trace conjecture} \cite{BMM} \emph{The  algebra $\mathcal{H}(W)$ admits a canonical symmetrising trace.}\\

The canonicity of the symmetrising trace  consists of the satisfaction of three conditions, which are given later in this article (see Conjecture \ref{BMM sym}) and we do not repeat here. Its existence gives rise to the definition of the Schur elements for $\mathcal{H}(W)$, which are certain Laurent polynomials that control the modular representation theory of $\mathcal{H}(W)$. They are also ubiquitous in the study of the families of characters for complex reflection groups, because they correspond to the inverses of the generic degrees, but also because the families of characters in the complex case have been defined as blocks of the Hecke algebra, following \cite{Rou} (see \cite{chlou} for the approach to this topic with the use of Schur elements). Contrary to the BMR freeness conjecture, the BMM symmetrising trace conjecture  remains unsolved. Our contributions towards its proof are the articles \cite{BCCK} and \cite{BCC}, where we prove the conjecture for the rank $2$ exceptional groups $G_4,G_5,G_6,G_7,G_8$ and $G_{13}$ respectively ($G_4$ was already known by \cite{MM10} and by  \cite{Mar46}), using a combination of computer algorithms.

Now, in order for everything to make sense and work as in the real case when it comes to induction from parabolic Hecke subalgebras, two fundamental conjectures, which were also standard facts in the real case, had to be stated. They are the parabolic counterparts of the first two conjectures:

$ $\\
\textbf{The parabolic freeness conjecture} \cite{malrou} \emph{Let $W'$ be a parabolic subgroup of $W$. The algebra $\mathcal{H}(W)$  is free as a left and right $\mathcal{H}(W')$-module of rank $|W|/|W'|$.}\\

One can notice that the BMR freeness conjecture is a particular case of the parabolic freeness conjecture, if $W'$ is taken to be the trivial group. Moreover, having a basis for $\mathcal{H}(W')$ as an $R(W')$-module and a basis for $\mathcal{H}(W)$ as an $\mathcal{H}(W')$-module allows the determination of a basis for $\mathcal{H}(W)$ as an $R(W)$-module. This approach has been used  sometimes by  people who have tackled the BMR freeness conjecture for the exceptional irreducible complex reflection groups,
 in order to simplify the calculations.

$ $\\
{\textbf{The parabolic trace conjecture} \cite{BMM} \emph{Let $W'$ be a parabolic subgroup of $W$. Let $\tau$ be the canonical symmetrising trace on $\mathcal{H}(W)$. The restriction $\tau|_{\mathcal{H}(W')}$ is the canonical symmetrising trace on $\mathcal{H}(W')$.}}\\

The second conjecture is in fact the last part of \cite[2.1, Assumption 2]{BMM}, whose first part is the classical BMM symmetrising trace conjecture. Both conjectures appeared later in \cite{malrou}, where Malle and Rouquier assumed their validity in order to study families of characters for complex reflection groups using Rouquier's definition and truncated induction from parabolic subgroups. Another application of the parabolic trace conjecture is that it allows the determination of the Schur elements of $\mathcal{H}(W')$ from those of $\mathcal{H}(W)$, or the other way round (see, for example,  \cite[Lemma 2.3.5]{chlou}).

Until now, the validity of both conjectures has been an open problem for any non-real complex reflection group.
In this paper we prove them for all  complex reflection groups of rank $2$ for which the BMM symmetrising trace conjecture is known to hold (we prove the parabolic freeness conjecture for a couple extra). The idea is the following: in \cite{BCCK} and \cite{BCC}, our method for proving
 the BMM symmetrising  trace conjecture for $G_n$, where $n \in \{4,5,6,7,8,13\}$, relied on the choice of a ``good'' basis $\mathcal{B}(W)$ for $\mathcal{H}(G_n)$ as an $R(G_n)$-module, 
 so that
\begin{itemize}
	\item[(i)] $1 \in \mathcal{B}(W)$; \smallbreak
	\item[(ii)] $\mathcal{B}(W)$ specialises to $W$  when $\mathcal{H}(W)$ specialises to the group algebra of $W$; \smallbreak
	\item[(iii)]  $\tau(b)=\delta_{1b}$ for all $b \in \mathcal{B}(W)$.
\end{itemize}	
We then created a \texttt{C++} program that expressed every product of a generator of $\mathcal{H}(W)$ belonging to $\mathcal{B}(W)$ with an element of $\mathcal{B}(W)$ as a linear combination of elements of $\mathcal{B}(W)$. This led us to the creation of an algorithm that expresses any element of  $\mathcal{H}(W)$ as a linear combination of elements of $\mathcal{B}(W)$, which we present here in Section \ref{gap3section}. We used the programming language GAP3 for the implementation of the algorithm, and the program that we created can be found on the project's webpage \cite{web}. Now, using this program, we can take a set of $|W|$ elements of $\mathcal{H}(W)$ and check whether it is a basis by calculating the supposed change of basis matrix: if the determinant of this matrix is a unit in $R(W)$, then the matrix is indeed a change of basis matrix and the set in question is a basis of  $\mathcal{H}(W)$ as an $R(W)$-module.

Having this program in our hands, we decided to solve the parabolic freeness conjecture for the groups  $G_n$ above. One of the perks when working with rank $2$ groups is that all parabolic subgroups are cyclic groups. As we will see later in this paper, it is enough to prove the conjecture for the subgroups generated by the generators of $W$. Therefore, if $s$ is a generator of $\mathcal{H}(W)$ and $W'$ is the corresponding cyclic parabolic subgroup of $W$, we need to find a subset $\mathfrak{B}^l_s(W)$ containing $|W|/|W'|$ elements of $\mathcal{H}(W)$ so that 
the set  $\mathfrak{P}_s^l(W):=\{s^jb\,|\, j=0,\ldots, |W'|-1,\,b \in \mathfrak{B}_s^l(W)\}$
is a basis of $\mathcal{H}(W)$ as an $R(W)$-module. 
Then $\mathfrak{B}_s^l(W)$ is a basis of $\mathcal{H}(W)$ as a left $\mathcal{H}(W')$-module.
We call the set  $\mathfrak{P}_s^l(W)$ a \emph{left parabolic basis of $\mathcal{H}(W)$ with respect to the generator $s$}. We can similarly define $\mathfrak{B}_s^r(W)$ and $\mathfrak{P}_s^r(W)$ when considering $\mathcal{H}(W)$ as a right $\mathcal{H}(W')$-module. 

In Section \ref{Finding parabolic bases}, we explain how we came up with parabolic bases for $G_4$, $G_7$, $G_8$ and $G_{13}$, using the Etingof--Rains surjection, as the first author did in \cite{Ch17} in order to find bases for the generic Hecke algebras of the groups $G_4,\ldots,G_{16}$. The idea is to choose a good expression for every element of $G_n$ and consider the corresponding element inside $\mathcal{H}(G_n)$. Not all choices we tried were good, but this is why having the GAP3 program was extremely useful. We could proceed through trial and error, especially in the case $G_{13}$, where we had sometimes many options for the same element. For $G_5$ and $G_6$, we show that the parabolic bases can be deduced from the ones for $G_7$. Finally, for some other rank $2$ groups, we explain why the bases existing already in literature prove the validity of the parabolic freeness conjecture. We conclude the following:

\begin{thm}
The parabolic freeness conjecture holds for the exceptional irreducible complex reflection groups $G_4$, $G_5$, $G_6$, $G_7$, $G_8$,  $G_{12}$, $G_{13}$, $G_{14}$, $G_{16}$ and $G_{22}$.
\end{thm}

We also point out here that all the parabolic bases that we give in this paper are good in the sense described above (that is, they satisfy conditions (i), (ii) and (iii)).  
Now, in Section \ref{parabtra}, we give a uniform proof of the parabolic trace conjecture for the groups above for which the BMM symmetrising trace conjecture is known to hold. In that regard, we also show that the BMM symmetrising trace conjecture  
holds for cyclic groups. We thus have the following:

\begin{thm}
The parabolic trace conjecture holds for the exceptional irreducible complex reflection groups $G_4$, $G_5$, $G_6$, $G_7$, $G_8$,  $G_{12}$, $G_{13}$ and $G_{22}$.
\end{thm}

To conclude, we would like to thank the following people, who have worked with reflection groups and their parabolic subgroups and gave us a lot of useful information on the topic: Meinolf Geck, Thomas Gobet, Ivan Marin, Jean Michel, Christoph Sch\"onnenbeck and Don Taylor.

\section{Hecke algebras and parabolic Hecke subalgebras}
In this section we will define and discuss some properties of the algebraic objects that we study.

\subsection{Reflection groups}\label{rc}
Let $V$ be a finite dimensional $\mathbb{R}$-vector space.  A {\em real reflection group} is a finite subgroup of $\mathrm{GL}(V)$ generated by \emph{reflections}, that is,
elements of  $\mathrm{GL}(V)$ of order 2  whose fixed points in $V$ form a hyperplane. 
Coxeter \cite{Cox1} proved that every real reflection group $W$ admits a presentation as follows: $$W=\langle s\in S\;|\; (st)^{m_{st}}=1
\text{ for all } s,t \in S\rangle$$
where $m_{st}\in \mathbb{Z}_{\geq2}$ for $s\not=t$ and $m_{ss}=1$.  
On the other hand, as Coxeter also proved \cite{Cox2}, every finite group with such a presentation, that is, every {\em finite Coxeter group},  is a real reflection group. The tuple $(W,S)$ is called a {\em finite Coxeter system} and the elements of $S$ are known as \emph{simple reflections}.

A real reflection group $W$ is called \emph{irreducible} if it acts irreducibly on $V$, \emph{i.e.,} $V$ does not admit any proper $W$-invariant subspace.   
Since every real reflection group is a direct product of irreducible ones, one can restrict the study of  real reflection groups to the study of the irreducible ones. 
If $W$ is an irreducible real reflection group, then the \emph{rank} of $W$ is the dimension of $V$.
 Coxeter \cite{Cox2} classified (up to isomorphism) all  irreducible finite Coxeter groups: they consist of three one-parameter families of increasing rank (also known as \emph{classical types}), denoted by $A_n, B_n, D_n$ (with $A_n \cong \mathfrak{S}_{n+1}$, $B_n \cong (\Z/2\Z)^n \rtimes \mathfrak{S}_n$, $D_n \cong (\Z/2\Z)^{n-1} \rtimes \mathfrak{S}_n$), an one-parameter family of rank two groups denoted by $I_2(m)$ (these are the dihedral groups, with $I_2(3)=A_2$ and $I_2(4)=B_2$) and six exceptional groups, denoted by $E_6, E_7, E_8, F_4, H_3, H_4$.
The groups of type $A_n$, $B_n$, $D_n$, $E_6$, $E_7$, $E_8$, $F_4$ and $I_2(6)$ are known as {\em Weyl groups}, and they are exactly the ones for which we have $m_{st}\in\{2,3,4,6\}$ in the Coxeter presentation.

All finite Coxeter groups are particular cases of complex reflection groups. Let $V$ be a finite dimensional $\mathbb{C}$-vector space. A {\em complex reflection group} is a finite subgroup of $\mathrm{GL}(V)$ generated by \emph{pseudo-reflections}, that is,
non-trivial elements of  $\mathrm{GL}(V)$ whose fixed points in $V$ form a hyperplane, called the \emph{reflecting hyperplane} of the given pseudo-reflection.
As in the real case, a complex reflection group $W$ is \emph{irreducible} if it acts irreducibly on $V$ and, if that is the case, the  \emph{rank} of $W$ is
the dimension of $V$. The classification of irreducible complex reflection groups is due to  Shephard and Todd \cite{ShTo} and given by the following theorem.

\begin{thm}\label{ShToClas} Let $W \subset \mathrm{GL}(V)$ be an irreducible complex
	reflection group. Then one of
	the following assertions is true:
	\begin{itemize}
	\item There exist  $d,e,n \in \N^*$ with  $(de,e,n)\neq(2,2,2)$ such
		that $(W,V) \cong (G(de,e,n),\C^{n-\delta_{de,1}})$, where $G(de,e,n)$ is the group of all 
		$n \times n$ monomial matrices whose non-zero entries are ${de}$-th roots of unity, while the product of all non-zero
		entries is a $d$-th root of unity. \smallbreak
		\item $(W,V)$ is isomorphic to one of the 34 exceptional groups
		$G_n$, with $n=4,\ldots,37$ (ordered with respect to increasing rank).
	\end{itemize}
\end{thm}

Among the irreducible complex reflection groups we encounter the irreducible finite Coxeter groups. More precisely, $G(1,1,n) \cong A_{n-1}$, $G(2,1,n) \cong B_n$,  $G(2,2,n) \cong D_n$, $G(m,m,2) \cong I_2(m)$, 
$G_{23} \cong H_3$,  $G_{28}  \cong  F_4$, $G_{30}  \cong H_4$, $G_{35}  \cong  E_6$, $G_{36}  \cong  E_7$, $G_{37}  \cong E_8$.

We know by \cite[Theorem 0.1]{Bes2} that every complex reflection group admits a Coxeter-like presentation. The generators of this presentation are pseudo-reflections and the relations are of two types: there are the relations that give the order of the pseudo-reflections, and there are also some homogeneous relations between positive words in the generating elements (that is, equalities between words of the same length). We call the latter \emph{braid relations}.

Let $W\subset GL(V)$ be a complex reflection group. We call the \emph{field of definition} of $W$, and denote by $K(W)$, the field generated by the traces on $V$ of all the elements of $W$. Benard \cite{Ben} and Bessis \cite{Bes1} have proved that $K(W)$ is a splitting field for $W$. If $K(W) \subseteq \mathbb{R}$, then $W$ is a finite Coxeter group, and
if $K(W)=\mathbb{Q}$, then $W$ is a Weyl group.

\subsection{Parabolic subgroups}\label{parr1}

First let us  consider the case where $W\subset GL(V)$ is a real reflection group.
Let $(W,S)$ be a finite Coxeter system, let $J\subset S$ and let $W_J$ be the subgroup of $W$ generated by $J$. We call $W_J$ a \emph{standard parabolic subgroup} of $W$. A subgroup $W'$ of $W$ is a \emph{parabolic subgroup} if it is conjugate to a standard parabolic subgroup $W_J$.
Given that every reflection in $W$ is conjugate to an element of $S$, parabolic subgroups are examples of \emph{reflection subgroups}, that is, subgroups of $W$ that are generated by
reflections. Note, however, that not all reflection subgroups of $W$ are parabolic subgroups. For example, $D_n$ is a reflection subgroup of $B_n$, but not a parabolic one.

Obviously, all reflection subgroups of $W$ are Coxeter groups as well. A canonical way of identifying a Coxeter presentation for any reflection subgroup of $W$ is given in \cite{Deo, Dy87, Dy90}. 
In the simpler case of standard parabolic subgroups,
 we have that $(W_J, J)$ is a finite Coxeter system (see \cite{Dy90} and \cite[\S 1.29]{gepf}). This implies that $(wW_Jw^{-1}, wJw^{-1})$ is also a finite Coxeter system for all $w \in W$, whence we obtain a Coxeter presentation for any parabolic subgroup of $W$.

The  following result, which is a corollary of \cite[6, Proposition 1]{Bou05}, yields a characterisation of parabolic subgroups of finite Coxeter groups as pointwise stabilisers of subspaces of $V$.

\begin{prop}\label{par1}
Let $W\subset GL(V)$ be a real reflection group and let $X$ be a subspace of $V$. Then the group
$$Fix(X):=\{w\in W\;|\; w(x)=x  \text{ for all } x \in X\}$$
is a parabolic subgroup of $W$, generated by the simple reflections whose reflecting hyperplane contains $X$. Conversely, if $W'$
 is a a parabolic subgroup of $W$, then there exists a subspace $X\subset V$ such that 
	$W'=Fix(X)$.
	\end{prop}
 
The alternative definition for parabolic subgroups given by the proposition above  can be also applied to the complex case. More precisely, let $W\subset \mathrm{GL}(V)$ be a complex reflection group. Let $X$ be a subspace of $V$ and set $W_X:=Fix(X)$. 
We call 
$W_X$ a \emph{parabolic subgroup} of $W$. 

\begin{rem}\rm
Often in bibliography $X$ is simply taken to be a subset of $V$. This does not affect the definition of parabolic subgroups, since $Fix(X)=Fix({{\rm Span}(X)})$, where
${\rm Span}(X)$ denotes the linear span of $X$. 
\end{rem}

As in the real case, we have that $W_X$ is a complex reflection group, thanks to the following result by Steinberg 
\cite[Theorem 1.5]{St64} (see also \cite{Le} for a shorter proof):

\begin{thm}\label{par2}
	Let $W\subset\mathrm{GL}(V)$ be a complex reflection group and let $X$ be a subspace of $V$. The parabolic subgroup $W_X$ of $W$ is generated by the pseudo-reflections of $W$ whose reflecting hyperplane contains $X$.
\end{thm}

\begin{exmp}\rm
We have $W_V=\{1\}$ and $W_{\{0\}}=W$.
\end{exmp}

\begin{exmp}\label{minimal parabolic}\rm  
If $H$ is a reflecting hyperplane of $W$, then the parabolic subgroup $W_H$ 
is isomorphic to a finite subgroup of $\mathbb{C}^{\times}$, whence $W_H$ is cyclic.
Thus, $W_H$ is a minimal non-trivial parabolic subgroup of $W$, whose non-trivial elements are pseudo-reflections. 
\end{exmp}

The following result is a straightforward corollary of Theorem \ref{par2} and it allows us to restrict to the case where $X$ is an intersection of reflecting hyperplanes.

\begin{cor}
Let $W\subset\mathrm{GL}(V)$ be a complex reflection group and let $X$ be a subspace of $V$. Let $I$ denote the intersection of all reflecting hyperplanes of $W$ that contain $X$. Then $W_X=W_I$.
\end{cor}

By  \cite[Theorem 2.1]{mut}, in order to find the parabolic subgroups of a complex reflection group,  it is enough to determine the parabolic subgroups of irreducible complex reflection groups. Following \cite[Lemma 3.3 and Theorem 3.11]{taylor} (see also \cite[Theorem 3.6]{mut}), $W'$ is a parabolic subgroup of $G(de,e,n)$ if and only if $W'$ is isomorphic to a direct product of the form
$$G(de,e,n_0) \times \mathfrak{S}_{n_1} \times \mathfrak{S}_{n_2} \times \cdots \times
\mathfrak{S}_{n_k}$$  
with $n_0, n_1, \ldots,n_k \in \N$ and $\sum_{j=0}^k n_j = n$ (if $n_j=0$, then the corresponding factor is omitted). As far as the exceptional groups are concerned, a list of all parabolic subgroups (up to conjugation)  for each group $G_n$, $n=4,\dots, 37$ is given in  \cite[Appendix C]{OT} and in \cite{ta}. In the latter, Taylor uses MAGMA to construct all reflection subgroups for each $G_n$ and then  checks which of them are the pointwise stabilisers of their space of fixed points.

\subsection{Braid groups}\label{br} Let $W\subset\mathrm{GL}(V)$ be a complex reflection group. Let $\mathcal{R}$ denote the set of pseudo-reflections of $W$, and let $\mathcal{A}$ denote the set of reflecting hyperplanes of $W$. Set $\mathcal{M}:=V\setminus \cup_{H\in \mathcal{A}}H$. 
As shown in \cite[\S2B]{BMR}, we can always  restrict to the case where
$W$ is \emph{essential}, meaning that $\cap_{H\in \mathcal{A}}H=\{0\}$. 
Steinberg  \cite[Corollary 1.6]{St64} proved that the action of $\mathcal{M}$ on $X$ is free.
Therefore, it defines a Galois covering $\mathcal{M}\rightarrow \mathcal{M}/W$, which gives rise to the following exact sequence for every $x\in \mathcal{M}$:
 \begin{center}
\begin{tikzcd}
	1 \arrow[rightarrow]{r} 
	& \pi_1(\mathcal{M},x)\arrow[rightarrow]{r}
	& \pi_1(\mathcal{M}/W, \underline{x})\arrow{r}
	& W\arrow{r}
	& 1,
\end{tikzcd}
\end{center}
where $\underline{x}$ denotes the image of $x$  under the canonical surjection $ \mathcal{M}\rightarrow \mathcal{M}/W$. 

Let now $x$ be some fixed basepoint of $\mathcal{M}$. 
In \cite[\S 2]{BMR}, Brou\'e, Malle and Rouquier  defined the {\em pure braid group} $P(W)$ and the 
{\em braid group} $B(W)$ of $W$ as $P(W):=\pi_1(\mathcal{M},x)$ and $B(W):=\pi_1(\mathcal{M}/W, \underline{x})$ respectively. Moreover, they associated to every element of $\mathcal{R}$ homotopy classes in $B(W)$ that we call \emph{braided reflections}. By  \cite[Theorem 0.1]{Bes2}, the braid group $B(W)$ admits a presentation with the generators being braided reflections and the relations being homogeneous relations between positive words in the generating elements, that is, braid relations.
In fact, the Coxeter-like presentation of $W$ is derived from this (Artin-like) presentation of $B(W)$, thanks to the following short exact sequence:
 \begin{center}
\begin{tikzcd}
	1 \arrow[rightarrow]{r} 
	& P(W)\arrow[rightarrow]{r}
	& B(W)\arrow{r}
	& W\arrow{r}
	& 1.
\end{tikzcd}
\end{center}

From now on, we will denote by $Z(G)$ the centre of any group $G$.
If $W$ is an irreducible complex reflection group, then $Z(W)$ is in bijection with a finite subgroup of $\C^{\times}$, and it is thus a cyclic group. 
We denote by $\boldsymbol{\pi}$ and $\boldsymbol{\beta}$ the homotopy classes of the loops $t\mapsto x \,{\rm exp}(2\pi i t)$ and $t\mapsto  x\,{\rm exp}(2\pi i t/|Z(W)|)$ respectively. Brou\'e, Malle and Rouquier \cite[Lemma 2.22 (2)]{BMR} proved that  $\boldsymbol{\pi} \in Z(P(W))$ and $\boldsymbol{\beta} \in Z(B(W))$, and they conjectured that each of the two elements  generates the centre it belongs to as a cyclic group. The fact that $Z(P(W)) = \langle \boldsymbol{\pi} \rangle$ was proved by Digne, Marin and Michel \cite[Theorem 1.2]{DMM}, while it was Bessis \cite[Theorem 12.8]{Bes3} who proved that
$Z(B(W)) = \langle \boldsymbol{\beta} \rangle$.

\subsection{Hecke algebras} 
Let $H \in \mathcal{A}$. We have seen in Remark \ref{minimal parabolic} that the minimal parabolic subgroup $W_H$ is cyclic. The generator of $W_H$ whose
only nontrivial eigenvalue is equal to ${\rm exp}(2\pi i/|W_H|)$ is called a \emph{distinguished} pseudo-reflection. We have the following result \cite[Corollary 2.38]{BCM}:

\begin{prop}\label{conjtogen}
If $S$ is a  generating set of distinguished pseudo-reflections for $W$, then any distinguished pseudo-reflection of $W$ is conjugate to an element of $S$.
\end{prop}

It is easy to see that if two pseudo-reflections are conjugate in $W$, then their reflecting hyperplanes belong to the same orbit under  the action of $W$ on $\mathcal{A}$. On the other hand, if $H'=wH$ for some $w \in W$, then 
$W_{H'} = wW_{H}w^{-1}$.
For every orbit $\mathcal{C} \in \mathcal{A}/W$, let
$e_{\mathcal{C}}$ denote the common order of the minimal parabolic subgroups $W_H$, where $H$
is any element of $\mathcal{C}$. 
Let 
$R(W):=\mathbb{Z}[\textbf{u},\textbf{u}^{-1}]$ denote the Laurent polynomial ring
in a set of indeterminates $\textbf{u}=(u_{\mathcal{C},j})_{(\mathcal{C} \in
	\mathcal{A}/W)(1\leq j \leq e_{\mathcal{C}})}$. The \emph{generic
	Hecke algebra}  $\mathcal{H}(W)$ of $W$ is  the quotient of the group
algebra $R(W)[B(W)]$ by the ideal
generated by the elements of the form
\begin{equation}\label{Hecker}
(s-u_{\mathcal{C},1})(s-u_{\mathcal{C},2}) \cdots (s-u_{\mathcal{C},e_{\mathcal{C}}}),
\end{equation}
where $\mathcal{C}$ runs over the set $\mathcal{A}/W$ and
$s$ runs over the set of braided reflections whose images in $W$ have reflecting hyperplanes in $\mathcal{C}$. We have
$$(s-u_{\mathcal{C},1})(s-u_{\mathcal{C},2}) \cdots (s-u_{\mathcal{C},e_{\mathcal{C}}})=0 \Leftrightarrow
{s}^{e_{\mathcal{C}}}-a_{{\mathcal{C}},e_{\mathcal{C}-1}}s^{e_{\mathcal{C}}-1}-a_{\mathcal{C},e_{\mathcal{C}-2}}{s}^{e_{\mathcal{C}}-2}-\dots-a_{{\mathcal{C}},0}=0,$$
where $a_{{\mathcal{C}},e_{\mathcal{C}}-j}:=(-1)^{j-1}f_j(u_{{\mathcal{C}},1},\dots,u_{\mathcal{C},e_{\mathcal{C}}})$ with $f_j$ denoting the $j$-th elementary symmetric polynomial, for $j=1,\ldots,e_{\mathcal{C}}$. 
Therefore, in the presentation of $\mathcal{H}(W)$, we have 
the images of the braided reflections of $B(W)$ as generators, and
two kinds of relations:
the \emph{braid relations}, coming from the Artin-like presentation of $B(W)$, and the \emph{positive Hecke relations}: 
\begin{equation}\label{Heckerp}
{s}^{e_{\mathcal{C}}}=a_{{\mathcal{C}},e_{\mathcal{C}-1}}s^{e_{\mathcal{C}}-1}+a_{\mathcal{C},e_{\mathcal{C}-2}}{s}^{e_{\mathcal{C}}-2}+\dots+a_{{\mathcal{C}},0}.
\end{equation}

We notice now that $a_{{\mathcal{C}},0}= (-1)^{e_{\mathcal{C}-1}}u_{{\mathcal{C}},0}u_{{\mathcal{C}},1}\dots u_{{\mathcal{C},e_{\mathcal{C}-1}}} \in R(W)^{\times}$. Hence,  $s$ is invertible in $\mathcal{H}(W)$ with
\begin{equation}\label{invhecke}
{ s}^{-1}=a_{\mathcal{C},0}^{-1}\,{s}^{e_{\mathcal{C}}-1}-a_{\mathcal{C},0}^{-1}\,a_{{\mathcal{C}},e_{\mathcal{C}}-1}s^{e_{\mathcal{C}}-2}-a_{\mathcal{C},0}^{-1}\,a_{{\mathcal{C}},e_{\mathcal{C}}-2}{s}^{e_{\mathcal{C}}-3}-\dots- a_{\mathcal{C},0}^{-1}\,a_{{\mathcal{C}},1}.
\end{equation}
We call  relations \eqref{invhecke} the \emph{inverse Hecke relations}. 

Finally, we define a \emph{word} in the generators of $\mathcal{H}(W)$ to be any product of the generators of $\mathcal{H}(W)$ or their inverses. If a word does not include inverses of the generators, we call it a \emph{positive word}.

\subsection{Parabolic Hecke subalgebras}
Let $I$ be an intersection of reflecting hyperplanes of $W$. We set $\mathcal{A}_I:=\{H\in \mathcal{A}\;|\; I\subset H\}$ and $\mathcal{M}_I:=V\setminus \cup_{H\in \mathcal{A}_I}H$. We fix as before an element $x\in \mathcal{M}$ and we set $P_I(W):=\pi_1(\mathcal{M}_I, x)$ and $B_I(W):=\pi_1(\mathcal{M}_I/W, \underline{x})$. We can then define a morphism of short exact sequences (see \cite[\S 2.D]{BMR})

\begin{center}
\begin{tikzcd}
	1 \arrow[rightarrow]{r} 
	& P_I(W)\arrow[rightarrow]{r}\arrow[hookrightarrow]{d}
	& B_I(W)\arrow{r}\arrow[hookrightarrow]{d}
	& W_I\arrow{r}\arrow[hookrightarrow]{d} 
	& 1\\
	1 \arrow[rightarrow]{r} 
	& P(W)\arrow[rightarrow]{r}
	& B(W)\arrow{r}
	& W\arrow{r}
	& 1
\end{tikzcd}
\end{center}
where all vertical arrows are injective and the embedding $B_I(W)\hookrightarrow B(W)$ is well-defined up to $P(W)$-conjugation. Moreover, this embedding sends braided reflections to braided reflections.

Let $W_I$ be the corresponding parabolic subgroup of $W$ and let $\textbf{v}=(v_{\mathcal{C},j})_{(\mathcal{C} \in
	\mathcal{A}_I/W_I)(1 \leq j \leq e_{\mathcal{C}})}$ be a set of indeterminates. 	
	We denote by $R_I(W)$ the Laurent polynomial ring $\mathbb{Z}[\textbf{v},\textbf{v}^{-1}]$.
Let $\alpha: \mathcal{A}_I/W_I\rightarrow \mathcal{A}/W$ be the map that sends a $W_I$-orbit of hyperplanes onto the corresponding $W$-orbit. We can then define a  ring morphism
$\varphi:\;R_I(W)\rightarrow R(W)$ given by $\varphi(v_{\mathcal{C},j})=u_{\alpha(\mathcal{C}),j}$ for $j=1,\ldots,e_{\mathcal{C}}$.

Now, the embedding $B_I(W)\hookrightarrow B(W)$ induces an inclusion (up to conjugation)
\begin{equation}\label{phitensor}
\mathcal{H}(W_I)\otimes_{\varphi}R(W)\hookrightarrow \mathcal{H}(W)
\end{equation}
whose image is called the \emph{parabolic Hecke subalgebra} of $\mathcal{H}(W)$ associated with $I$. We denote this algebra by $\mathcal{H}_I(W)$. To put it simply, the parabolic Hecke subalgebra $\mathcal{H}_I(W)$ is the Hecke algebra $\mathcal{H}(W_I)$ of $W_I$ defined over the ring  $R(W)$.

\section{The freeness and trace conjectures for Hecke algebras}

In this section, we will first discuss the two fundamental conjectures on the structure of Hecke algebras associated with complex reflection groups: the ``freeness conjecture'', stated by Brou\'e, Malle and Rouquier \cite{BMR}, and the ``(symmetrising) trace conjecture'', stated by Brou\'e, Malle and Michel \cite{BMM}. The first one is now a theorem, the second one is still open for most complex reflection groups. In \cite{BCCK} and \cite{BCC}, using computational methods, we proved the trace conjecture for several  exceptional groups of rank $2$. Our methods rely on finding a ``good'' basis for the Hecke algebra associated with each group.
At the end of this section, we present a GAP3 program that expresses any element of the Hecke algebra as a linear combination of the elements of our good basis. This in turn helps us to find more bases for the Hecke algebra by computing the change of basis matrix.

From now on, let $W\subset\mathrm{GL}(V)$ be a complex reflection group and let $\mathcal{H}(W)$ be the generic Hecke algebra of $W$.

\subsection{Bases for Hecke algebras}
If $W$ is a real reflection group, then $\mathcal{H}(W)$ admits a standard basis $(T_w)_{w \in W}$ indexed by the elements of $W$  \cite[IV, \S 2]{Bou05}. 
Brou\'e, Malle and Rouquier conjectured that a similar result holds for non-real complex reflection groups. More specifically, they stated the following \cite[\S 4]{BMR}: 

\begin{conj}[The BMR freeness conjecture]\label{BMR free}
	The algebra $\mathcal{H}(W)$ is a free $R(W)$-module of rank  $|W|$.
\end{conj}

Note that, thanks to the following result, which can be found in 	\cite[Proof of Theorem 4.24]{BMR} (for another detailed proof, one may also see \cite[Proposition 2.4]{Mar43}), 
	in order to prove the BMR freeness conjecture, it is enough to find a spanning set of $\mathcal{H}(W)$ as an $R(W)$-module consisting of $|W|$ elements.
	\begin{thm}\label{mention}
	 If $\mathcal{H}(W)$ is generated as $R(W)$-module by  $|W|$ elements, then it is a free 
	 $R(W)$-module of rank  $|W|$.
	\end{thm}

As mentioned earlier, the BMR freeness conjecture is now a theorem. It was proved for:
\begin{itemize}
       \item the complex reflection groups of the infinite series $G(de,e,n)$ by \cite{ArKo, BM, Ar};\smallbreak
	\item the group $G_4$ by \cite{BM, Fun, Mar41, Ch18} (4 independent proofs); \smallbreak
	\item the group $G_{12}$ by  \cite{MaPf}; \smallbreak
	\item the groups $G_5,\ldots, G_{16}$ by \cite{Ch17, Ch18}; \smallbreak
	\item the groups $G_{17},\,G_{18},\,G_{19}$ by \cite{Tsu} (with a computer method applicable to all rank $2$ groups);\smallbreak
	\item the groups $G_{20}, \,G_{21}$ by \cite{MarNew};\smallbreak
	\item the groups $G_{22},\ldots, G_{37}$ by \cite{Mar41, Mar43, MaPf}. 
\end{itemize}

 \subsection{Good bases for Hecke algebras}\label{good bases}
 A symmetrising trace on an algebra is a trace map that induces a non-degenerate bilinear form. There exists a canonical symmetrising trace on the group algebra of $W$ given by $\tau(w)=\delta_{1w}$. If $W$ is a real reflection group, then there exists a 
 canonical symmetrising on $\mathcal{H}(W)$ given by $\tau(T_w)=\delta_{1w}$  \cite[IV, \S 2]{Bou05}.
 Brou\'e, Malle and Michel  conjectured that all generic Hecke algebras possess a canonical symmetrising trace \cite[2.1, Assumption 2(1)]{BMM}:

\begin{conj}[The BMM symmetrising trace conjecture]
	\label{BMM sym}
	There exists a linear map $\tau: \mathcal{H}(W)\rightarrow R(W)$ such that:
	\begin{itemize}
		\item[$(1)$] $\tau$ is a symmetrising trace, that is, we have $\tau(h_1h_2)=\tau(h_2h_1)$ for all  $h_1,h_2\in \mathcal{H}(W)$, and the bilinear map $\mathcal{H}(W)\times \mathcal{H}(W)\rightarrow R(W)$,  $(h_1,h_2)\mapsto\tau(h_1h_2)$ is non-degenerate. \smallbreak
		\item[$(2)$] $\tau$ becomes the canonical symmetrising trace on $K(W)[W]$ when ${u}_{{\mathcal{C}},j}$ specialises to ${\rm exp}(2\pi i j/e_{\mathcal{C}})$ for every $\mathcal{C}\in \mathcal{A}/W$ and $j=1,\ldots, e_{\mathcal{C}}$. \smallbreak
		\item[$(3)$]  $\tau$ satisfies
		$$
		\tau(T_{b^{-1}})^* =\frac{\tau(T_{b\boldsymbol{\pi}})}{\tau(T_{\boldsymbol{\pi}})}, \quad \text{ for all } b \in B(W),
		$$
		where $b\mapsto T_{b}$ denotes the restriction of the natural surjection $R(W)[B(W)] \rightarrow \mathcal{H}(W)$ to $B(W)$ and
		$x \mapsto x^*$ the automorphism of  $R(W)$ given by $u_{\mathcal{C},j} \mapsto u_{\mathcal{C},j}^{-1}$.
		\smallbreak
	\end{itemize}
\end{conj}
Note that, by  \cite[2.1]{BMM}, since the BMR freeness conjecture holds,  there exists at most
one symmetrising trace satisfying Conditions (2) and (3) of Conjecture \ref{BMM sym}, meaning that $\tau$ is unique. We call $\tau$ the \emph{canonical symmetrising trace on  $\mathcal{H}(W)$}.

The BMM symmetrising trace conjecture has been proved for the following non-real complex reflection groups:
\begin{itemize}
         \item  the groups  $G(de,e,n)$ by \cite{BreMa, MM98} (with Condition (3) deriving from \cite[Lemma 2.7]{BMM});\smallbreak
	\item the group $G_4$ by \cite{MM10, Mar46, BCCK} (3 independent proofs); \smallbreak
	\item the groups $G_5, G_6, G_7, G_8$ by \cite{BCCK}; \smallbreak
	\item the group $G_{12}$ by \cite{MM10};  \smallbreak
	\item the group $G_{13}$ by \cite{BCC}; \smallbreak
	\item the groups $G_{22}, G_{24}$ by \cite{MM10}.
\end{itemize}	
In all the cases above, a \emph{good} basis $\mathcal{B}(W)$ for $\mathcal{H}(W)$ was considered, so that:
\begin{itemize}
	\item[(i)] $1 \in \mathcal{B}(W)$; \smallbreak
	\item[(ii)] $\mathcal{B}(W)$ specialises to $W$  when ${u}_{\mathcal{C},j}$ specialises to ${\rm exp}(2\pi i j/e_{\mathcal{C}})$
	for all $\mathcal{C}\in \mathcal{A}/W$ and $j=1,\ldots, e_{\mathcal{C}}$; \smallbreak
	\item[(iii)]  $\tau(b)=\delta_{1b}$ for all $b \in \mathcal{B}(W)$.
\end{itemize}	
This way, 	Condition (2) of Conjecture \ref{BMM sym} is satisfied, and only Conditions (1) and (3) have to be verified. 
In fact, Malle and Michel have conjectured that there is always a subset of $\mathcal{H}(W)$ (not necessarily a basis) that satisfies properties (i)--(iii) \cite[Conjecture 2.6]{MM10}:

\begin{conj}[The lifting conjecture]\label{lif}
	There exists a section $W \rightarrow \boldsymbol{W} \subset B(W)$, $w \mapsto \boldsymbol{w}$ of $W$ in $B(W)$ such that $1 \in \boldsymbol{W}$, and such that for any $\boldsymbol{w} \in \boldsymbol{W}$ we have $\tau(T_{\boldsymbol{w}}) = \delta_{1\boldsymbol{w}}$.
\end{conj}

If the lifting conjecture \ref{lif} holds, then Condition (2) of Conjecture \ref{BMM sym} is obviously satisfied. If further the elements $\{T_{\boldsymbol{w}}\,|\,\boldsymbol{w} \in \boldsymbol{W}\}$ form an
$R(W)$-basis of $\mathcal{H}(W)$, then, by \cite[Proposition 2.7]{MM10}, Condition (3) of Conjecture \ref{BMM sym} is equivalent to:
\begin{equation}\label{extra con2}
\tau(T_{\boldsymbol{w}^{-1}\boldsymbol{\pi}})=0, \quad \text{ for all } \boldsymbol{w} \in \boldsymbol{W} \setminus \{1\}.
\end{equation}

\subsection{More bases for Hecke algebras}\label{gap3section} From now on, we will write $b_1,b_2,\ldots,b_{|W|}$ for the elements of $\mathcal{B}(W)$, with $b_1=1$.
Checking Condition (1) of the BMM symmetrising trace conjecture amounts to showing that the Gram matrix $A(W):=(\tau(b_ib_j))_{1 \leq i,j \leq |W|}$ is symmetric and invertible over $R(W)$. By definition of $\tau$, $\tau(b_ib_j)$ is the coefficient  of $1$ when $b_ib_j$ is expressed as a linear combination of the elements of the basis $\mathcal{B}(W)$. This is why, in \cite{BCCK}, we created a program in the language \texttt{C++} which would write any product $b_ib_j$ as a linear combination of the elements of $\mathcal{B}(W)$. With the exception of $G_4$, this program was very time-consuming. Taking advantage of the fact that our basis $\mathcal{B}(W)$ has an inductive nature, we came up with an elaborate algorithm that calculates the entries of the matrix $A(W)$ row-by-row. For this, we only needed to use the  \texttt{C++} program to express $b_ib_j$ as a linear combination of the elements of the basis when $b_i \in \mathcal{B}(W)$ is a generator of $\mathcal{H}(W)$.
We then created a second program in SAGE \cite{sagemath}, which, using as inputs the outputs of  the \texttt{C++} program, produced very quickly the matrix $A(W)$ and its determinant. We were thus able to prove the BMM symmetrising trace conjecture for groups $G_4$, $G_5$, $G_6$, $G_7$ and $G_8$. In \cite{BCC}, we adapted this method to obtain the validity of the BMM symmetrising trace conjecture for the group $G_{13}$, thus completing its proof for all exceptional $2$-reflection groups of rank $2$ (the others being $G_{12}$ and $G_{22}$).

For the purposes of the current article and for future works, we have now created a GAP3 program, which can be found on \cite{web}, that can be used to express any element of the generic Hecke algebra $\mathcal{H}(W)$ as a linear combination of the elements of the basis $\mathcal{B}(W)$ for the groups that we studied in \cite{BCCK} and \cite{BCC}. 
More specifically, given that any element of $\mathcal{H}(W)$  is a linear combination of words in $\mathcal{H}(W)$, our program uses the outputs of the  \texttt{C++} program to express any word as a linear combination of the elements of  $\mathcal{B}(W)$. 

The algorithm of our GAP3 program is simple and follows the steps of the proof of the following result.

\begin{lem}\label{gap-al}
Let $\mathcal{S}$ be a subset of the generators of $\mathcal{H}(W)$ and assume that one can express $sb_j$ as a linear combination of elements of $\mathcal{B}(W)$ for all $s \in \mathcal{S}$ and for all $b_j \in \mathcal{B}(W)$. 
Let $h$ be a word in the generators in $\mathcal{S}$, that is, a product of elements of $\mathcal{S}$ or their inverses.
Then one can express $hb_j$ as a linear combination of elements of $\mathcal{B}(W)$ for all $b_j \in \mathcal{B}(W)$. In particular, one can  express $h$ as a linear combination of elements of $\mathcal{B}(W)$.
\end{lem}

\begin{proof}
Let $s \in \mathcal{S}$ and $b_j \in \mathcal{B}(W)$. We will denote by $\lambda^s_{j,k}$ the coefficient of $b_k$ in $R(W)$ when $sb_j$ is written as a linear combination of  elements of $\mathcal{B}(W)$, for $k=1,\ldots,|W|$. That is, we have
 $$sb_j = \sum_{k=1}^{|W|} \lambda^s_{j,k}b_k.$$

First, let $h$ be a positive word in $\mathcal{S}$, that is, a product of elements of $\mathcal{S}$. We will prove the desired result by induction on the length $\ell(h)$ of $h$, that is, the number of factors in this product.

If $\ell(h)=1$, then $h=s$ for some $s \in \mathcal{S}$, and we have $hb_j= \sum_{k=1}^{|W|} \lambda^s_{j,k}b_k$.
Now assume that the statement is true for all positive words of length $n$ and let
$\ell(h)=n+1$. Then $h=sh'$ for some $s \in \mathcal{S}$ and some positive word $h'$ in $\mathcal{S}$ with $\ell(h')=n$. 
By induction hypothesis, we know how to express $h'b_j$ as a linear combination of elements of $\mathcal{B}(W)$.
Hence, if $h' b_j= \sum_{k=1}^{|W|} \mu_{j,k} b_k$, with $\mu_{j,k} \in R(W)$, then
$$hb_j = sh'b_j = s \left(\sum_{k=1}^{|W|} \mu_{j,k} b_k \right) = 
\sum_{k=1}^{|W|} \mu_{j,k} (sb_k) = \sum_{k=1}^{|W|} \mu_{j,k} \left( \sum_{l=1}^{|W|} \lambda^s_{k,l}b_l \right) = \sum_{k=1}^{|W|}  \sum_{l=1}^{|W|} \mu_{j,k}  \lambda^s_{k,l}b_l.$$ 

If now $h$ is any word in $\mathcal{S}$, then we can use the inverse Hecke relations \eqref{invhecke} to write $h$ as a linear combination of positive words, then apply the above algorithm to each positive word to obtain $hb_j$ as a linear combination of elements of $\mathcal{B}(W)$.
\end{proof}

Now, let $W \in \{G_4,G_5,G_6,G_7,G_8,G_{13}\}$, and let $\mathcal{S}$ be the set of all generators of $\mathcal{H}(W)$ that belong to $\mathcal{B}(W)$. Given the outputs of our   \texttt{C++} program, Lemma \ref{gap-al} implies that we can express any product $hb_j$ as a linear  combination of elements of $\mathcal{B}(W)$, where $h$ is a word in $\mathcal{S}$ and $b_j \in \mathcal{B}(W)$.
For all groups except for $G_7$, the set $\mathcal{S}$ contains all generators of $\mathcal{H}(W)$, and thus we are able to express any word in $\mathcal{H}(W)$ as
a  linear  combination of elements of $\mathcal{B}(W)$. However, 
for group $G_7$, only two out of three generators of $\mathcal{H}(W)$ belong to $\mathcal{B}(W)$. Let $s$ be the generator that does not belong to $\mathcal{S}$, and let $j \in \{1,\ldots, |G_7|\}$. In \cite{BCCK}, right before Proposition 5.7, we explain how we can write $sb_j$ as a linear combination of elements of the form $hb_k$, where $h$ is a word in $\mathcal{S}$. Following the discussion above, we can thus express $sb_j$ as a linear combination of elements of $\mathcal{B}(W)$. Now, applying Lemma \ref{gap-al} to the set of all generators of $\mathcal{H}(G_7)$ allows us to express any word in $\mathcal{H}(G_7)$ as
a  linear  combination of elements of $\mathcal{B}(G_7)$.

The existence of our algorithm yields an alternative proof for the fact that $\mathcal{B}(W)$ is a spanning set for $\mathcal{H}(W)$ as an $R(W)$-module. By Theorem \ref{mention}, this implies that $\mathcal{B}(W)$ is a basis of $\mathcal{H}(W)$ over $R(W)$. As we remarked in both \cite{BCCK} and \cite{BCC}, one could argue that we have thus obtained a computerised proof of the BMR freeness conjecture for these groupes. This is indeed the case for $G_4$. However, many of the calculations 
made by hand in \cite{Ch17} for the proof of the BMR freeness conjecture were incorporated in the \texttt{C++} program.

Moreover, thanks to our GAP3 algorithm, we can now easily check  whether any subset of $\mathcal{H}(W)$ consisting of $|W|$ elements is a basis of $\mathcal{H}(W)$ as an $R(W)$-module. Let $\{b_1',\ldots,b_{|W|}'\}$  be such a set. Applying our algorithm, for all $j=1,\ldots,|W|$, we write $$b_j' = \sum_{i=1}^{|W|} \mu_{ij} b_i$$ with $\mu_{ij} \in R(W)$. We compute the determinant of the matrix $M:=(\mu_{ij})_{1\leq i,j \leq |W|}$ and if it is a unit in $R(W)$, then $M$ is a change of basis matrix and the set $\{b_1',\ldots,b_{|W|}'\}$  is a basis of $\mathcal{H}(W)$. We have used this method to prove the results in the last section of this article.

\section{The trace and freeness conjectures for parabolic Hecke subalgebras}

In this section, we will present two further conjectures, which are fundamental as far as the structure of parabolic Hecke subalgebras is concerned. They are both essential in order to be able to smoothly use induction and restriction arguments.
We will call them the ``parabolic freeness conjecture'' and the ``parabolic (symmetrising) trace conjecture'' because they have obvious similarities with the freeness and trace conjectures on the Hecke algebra level. Our objective in this paper was to prove both conjectures for the exceptional groups of rank $2$ that we studied in \cite{BCCK} and \cite{BCC}. Since we are able to provide a short uniform proof of the parabolic trace conjecture (which also works for groups $G_{12}$ and $G_{22}$), while we use a much longer case-by-case analysis for the parabolic freeness conjecture, we will present the two conjectures in, what is in fact, their chronological order.

From now on, let $W\subset\mathrm{GL}(V)$ be a complex reflection group and let $\mathcal{H}(W)$ be the generic Hecke algebra of $W$.   Let $I$ be an intersection of reflecting hyperplanes of $W$ and let $W_I$ be the corresponding parabolic subgroup of $W$. Let $\mathcal{H}_I(W)$ be the parabolic Hecke subalgebra of $\mathcal{H}(W)$ associated with $I$.

\subsection{The parabolic trace conjecture} \label{parabtra}
Let us  assume that the BMM symmetrising trace conjecture holds for the algebras $\mathcal{H}(W)$ and $\mathcal{H}(W_I)$. Then the algebra $\mathcal{H}_I(W)$ admits also a canonical symmetrising trace, which is obtained from the one of $\mathcal{H}(W_I)$ by specialisation of scalars (via the map $\varphi$, as in \eqref{phitensor}). The following conjecture was stated by Brou\'e, Malle and Michel together with the BMM symmetrising trace conjecture \cite[2.1, Assumption $2(2)(c)$]{BMM}:

\begin{conj}[The parabolic symmetrising trace conjecture]\label{conjp2}
	Let $\tau$ be the canonical symmetrising trace on $\mathcal{H}(W)$. The restriction $\tau|_{\mathcal{H}_I(W)}$ is the canonical symmetrising trace on $\mathcal{H}_I(W)$.
	\end{conj}

Conjecture \ref{conjp2} is obviously true for real reflection groups, because of the way the map $\tau$ is defined with the use of the standard basis $(T_w)_{w \in W}$. However, not much is known in the complex case. 

Let us consider the case where $I=H \in \mathcal{A}$. Then, as we have seen in Example \ref{minimal parabolic}, $W_H$ is cyclic and, thus, a minimal non-trivial parabolic subgroup of $W$. On the other hand, any cyclic parabolic subgroup of $W$ must be a pointwise stabiliser of a reflecting hyperplane. 
Every complex reflection group has cyclic parabolic subgroups, one for each reflecting hyperplane. In particular, if $W$ is a complex reflection group of rank $2$, then all its non-trivial proper parabolic subgroups are of rank $1$, and thus cyclic.

In order to prove Conjecture \ref{conjp2}  for complex reflection groups of rank $2$, we will need to explicitly define the canonical symmetrising trace on generic Hecke algebras of cyclic groups (in the Shephard--Todd classification, these correspond to the complex reflection groups $G(d,1,1)$).
For the sake of completeness and for the reader's convenience, we will first show that the BMM symmetrising trace conjecture holds for cyclic groups, verifying directly all three conditions of the conjecture.

\begin{prop} \label{l1}
	The BMM symmetrising trace conjecture holds for cyclic groups.
	\end{prop}
\begin{proof}
Let $W$ be a cyclic  group of order $d$, and let $R(W)=\mathbb{Z}[u_1^{\pm 1}, u_2^{\pm 1},\dots, u_{d}^{\pm 1}]$ for some indeterminates $u_1,u_2,\ldots,u_d$.
The algebra $\mathcal{H}(W)$ is the
	$R(W)$-algebra with presentation
	$$\mathcal{H}(W)=\langle s\;|\; s^d=a_{d-1}s^{d-1}+a_{d-2}s^{d-2}+\dots+a_1s+a_0\rangle$$
 where $a_{d-j}:=(-1)^{j-1}f_j(u_1,\ldots,u_d)$ with $f_j$ denoting the $j$-th elementary symmetric polynomial, for $j=1,\ldots,d$. In particular,
 $a_0=(-1)^{d-1} u_1u_2\dots u_d \in R(W)^\times$.
	
	Let $\mathcal{B}(W)=\{1,s,\dots, s^{d-1}\}$. Obviously, $\mathcal{B}(W)$ generates $\mathcal{H}(W)$ as an $R(W)$-module, and it is thus an $R(W)$-basis by Theorem \ref{mention}. We define the linear map $\tau: \mathcal{H}(W)\rightarrow R(W)$ given by $\tau(b)=\delta_{1b}$, for all $b\in \mathcal{B}(W)$. By definition, $\tau$  satisfies Condition (2) of Conjecture \ref{BMM sym}. As far as the Condition (1) is concerned, we obviously have $\tau(h_1h_2)=\tau(h_2h_1)$ for all $h_1, h_2 \in \mathcal{H}(W)$, since the algebra  $\mathcal{H}(W)$ is commutative.  For the second part of Condition (1), we will show that the determinant of the Gram matrix $A(W)=(\tau(bb'))_{b,b'\in \mathcal{B}(W)}$   is a unit in $R(W)$. Since  $\tau(s^{j-1})=\delta_{1j}$ for $j=1,\ldots,d$ and $s^d=a_{d-1}s^{d-1}+a_{d-2}s^{d-2}+\dots+a_1s+a_0$, the matrix $A(W)$ is a $d \times d$ matrix of the following form:
	$$A(W)=\begin{bmatrix}
	1&0&\cdots&0&0&0\\
	0&0&\cdots&0&0&a_0\\
	0&0&\cdots&0&a_0&*\\
	0&0&\cdots&a_0&*&*\\
	\vdots&\vdots&\udots&\udots&&\vdots\\
	0&a_0&*&\cdots&*&*
	\end{bmatrix}$$
	whence $\det(A(W))=(-1)^{\frac{(d-1)(d-2)}{2}}a_0^{d-1} \in R(W)^{\times}$.
	
	It remains now to prove Condition (3) of Conjecture \ref{BMM sym}.  Since the 
	lifting conjecture (Conjecture \ref{lif}) is satisfied by the basis $\mathcal{B}(W)$, it suffices to prove Condition \eqref{extra con2} instead. Since $T_{\boldsymbol{\pi}}=s^d$ in this case,  this amounts to proving that $\tau(s^{d-j})=0$, for every $j=1,\dots, d-1$. The result follows  from the definition of $\tau$, since $s^{d-j}\in \mathcal{B}(W)\setminus\{1\}$ for $j=1,\dots, d-1$.
	\end{proof}

\begin{cor}\label{l2}
	Let $W$ be a complex reflection group and let $H \in \mathcal{A}$. The 
	BMM symmetrising conjecture holds for $W_H$. In particular, $\mathcal{H}_H(W)$ admits a canonical symmetrising trace.
\end{cor}

The following corollary of Proposition \ref{l1} gives a criterion for the validity of Conjecture \ref{conjp2} with respect to minimal parabolic subgroups, and thus its overall validity for complex reflection groups of rank $2$.

\begin{cor}\label{l3}
Let $W$ be a complex reflection group such that $\mathcal{H}(W)$ admits a canonical symmetrising trace $\tau$.
Let $H \in \mathcal{A}$, and let $s$ be a generator of $\mathcal{H}_H(W)$. 
If 
\begin{equation}\label{crit}
\tau(s^j)=0 \text{ for all }j=1,\ldots,|W_H|-1,
\end{equation}
 then $\tau|_{\mathcal{H}_H(W)}$ is the canonical symmetrising trace on $\mathcal{H}_H(W)$.
\end{cor}

If now $H,H' \in \mathcal{A}$ belong to the same orbit under the action of $W$, then the parabolic subgroups $W_H$ and $W_{H'}$ are conjugate. Since $\tau$ is a trace map, Condition \eqref{crit} holds for $\mathcal{H}_H(W)$ if and only if it holds for 
$\mathcal{H}_{H'}(W)$. Hence, it is enough to verify its validity for only one hyperplane per hyperplane orbit.
Following Proposition \ref{conjtogen}, each hyperplane orbit includes the reflecting hyperplane of a generator of $W$. Therefore, if Condition  \eqref{crit}  is satisfied by every generator of $\mathcal{H}(W)$ (or at least one generator per conjugacy class), then the parabolic symmetrising trace conjecture holds for all minimal parabolic subgroups; in the particular case where $W$ is a complex reflection group of rank $2$, this yields the validity of the parabolic symmetrising trace conjecture.

\begin{thm}
The parabolic symmetrising trace conjecture holds for $G_4, G_5, G_6, G_7,  G_8, G_{12}, G_{13},G_{22}$.
\end{thm}

\begin{proof}
Let $W \in \{G_4,G_5,G_6,G_8,G_{12},G_{13},G_{22}\}$. Condition  \eqref{crit}  is satisfied by all generators of $\mathcal{H}(W)$, because all required powers of the generators belong to the good basis $\mathcal{B}(W)$ considered in \cite{BCCK,BCC,MM10}, for which we have $\tau(b)=\delta_{1b}$ for all $b \in \mathcal{B}(W)$ (in the case of $G_8$, this is true for one of the two generators, but the two generators are conjugate).
For $W=G_7$, this is also the case for the two generators that belong to $\mathcal{B}(G_7)$.  However, for the third generator $s$ (as denoted in \cite{BCCK}, later in this paper we denote it by $s_1$), we only need to check that $\tau(s)=0$. 
Since $s$ is written as a linear combination of elements of $\mathcal{B}(G_7)\setminus \{1\}$ (see \cite[\S 4.2.4]{BCCK}), we obtain that $\tau(s)=0$, as required.
\end{proof}

\subsection{The parabolic freeness conjecture}

The following conjecture was later stated by Malle and Rouquier in their article on families of characters for complex reflection groups \cite[Conjecture 2.1]{malrou}:
\begin{conj}[The parabolic freeness conjecture]\label{conjp1}
For each parabolic subgroup $W_I$ of $W$, the algebra $\mathcal{H}(W)$ is free as a left and right $\mathcal{H}_I(W)$-module of rank $|W|/|W_I|$.
\end{conj}

Sch\"onnenbeck \cite[Lemma 1.4.3]{Sch} has shown that if Conjecture \ref{conjp2} is true, then it is enough to prove Conjecture \ref{conjp1} only  ``from the left''.
However, as Marin pointed out to us with the following proposition, even without the validity of Conjecture \ref{conjp2}, it is enough to prove the parabolic freeness conjecture only ``from one side''.

\begin{prop}[Marin] \label{lr}
	For each parabolic subgroup $W_I$ of $W$, the algebra $\mathcal{H}(W)$ is free as a left  $\mathcal{H}_I(W)$-module of finite rank $r$ if and only if the algebra $\mathcal{H}(W)$ is free as a right $\mathcal{H}_I(W)$-module of finite rank $r$. Moreover, $r=|W|/|W_I|$.
\end{prop}

\begin{proof}
	Let $*:R(W)\rightarrow R(W)$ be the ring automorphism given by $u_{\mathcal{C},j}\mapsto u^{-1}_{\mathcal{C},j}$, which we have seen already in Conjecture \ref{BMM sym}(3). The group automorphism $B(W)\rightarrow B(W)^{op}$ defined by $b\mapsto b^{-1}$ induces an $R(W)$-algebra automorphism $\iota^{op}: R(W)[B(W)]\xrightarrow{\sim} (R(W)[B(W)])^{op}=R(W)[B(W)^{op}]$. 
	
Let now 	$\mathcal{J}$ denote the 2-sided ideal of $R(W)[B(W)]$ generated by the elements $\prod_{j=1}^{e_{\mathcal{C}}}(s-u_{\mathcal{C},j})$ as in \eqref{Hecker}. We have:
$$\iota^{op} \left( \prod\limits_{j=1}^{e_{\mathcal{C}}}(s-u_{\mathcal{C},j}) \right)
=\prod\limits_{j=1}^{e_{\mathcal{C}}}(s^{-1}-u_{\mathcal{C},j})=
\prod\limits_{j=1}^{e_{\mathcal{C}}}
	u_{\mathcal{C},j} s^{-1}(u_{\mathcal{C},j}^{-1}-s)=
(-s)^{-e_{\mathcal{C}}}
	u_{\mathcal{C},1}u_{\mathcal{C},2} \dots u_{\mathcal{C},e_{\mathcal{C}}}	  \prod\limits_{j=1}^{e_{\mathcal{C}}}(s-u_{\mathcal{C},j}^{-1})
$$
and so  $\iota^{op} (\mathcal{J})$ is the 2-sided ideal of $(R(W)[B(W)])^{op}$ generated by $\prod_{j=1}^{e_{\mathcal{C}}}(s-u^{-1}_{\mathcal{C},j})$.	
Hence, 	$\iota^{op}(\mathcal{J})=(\mathcal{J}^{op})^{*}$, where $\mathcal{J}^{op}$ denotes the 2-sided ideal of $R(W)[B(W)^{op}]$ generated by the elements
	$\prod_{j=1}^{e_{\mathcal{C}}}(s-u_{\mathcal{C},j})$.
	Since $R(W)[B(W)^{op}]/\mathcal{J}^{op}=\mathcal{H}(W)^{op}$ (see \cite[1.30(a)]{BMM}), we have:
	
	\begin{center}
		\begin{tikzcd}
		R(W)[B(W)^{op}] \arrow[rightarrow, "*"]{r}\arrow[twoheadrightarrow]{d}
		&	R(W)[B(W)^{op}] \arrow[rightarrow, "(\iota^{op})^{-1}"]{r}\arrow[twoheadrightarrow]{d}
		& R(W)[B(W)]\arrow[twoheadrightarrow]{d}\\
		 \mathcal{H}(W)^{op}\arrow[rightarrow, "\sim"]{r}
		& R(W)[B(W)^{op}]/(\mathcal{J}^{op})^*=R(W)[B(W)^{op}]/\iota^{op}(\mathcal{J})
		\arrow[rightarrow,"\sim"]{r}&\mathcal{H}(W)
		\end{tikzcd}
	\end{center}
	The composition $\grc: \mathcal{H}(W)^{op}\rightarrow \mathcal{H}(W)$ is a ring automorphism which maps $\mathcal{H}_I(W)^{op}$ onto $\mathcal{H}_I(W)$.
	
	Now assume that $\mathcal{H}(W)$ is a free left $\mathcal{H}_I(W)$-module of rank $r<\infty$. Then there exists a basis $\mathfrak{B}$ of $\mathcal{H}(W)$ with $|\mathfrak{B}|=r$  such that $\mathcal{H}(W) \cong \bigoplus_{b\in\mathfrak{B}}\mathcal{H}_I(W)b$ as $\mathbb{\Z}$-modules. We apply now the automorphism $\grc^{-1}$ and, using the fact that $\mathcal{H}(W)^{op} \cong \mathcal{H}(W)$ as $\mathbb{Z}$-modules, we obtain:
	$$\mathcal{H}(W) \cong \bigoplus_{b\in\mathfrak{B}}\mathcal{H}_I(W)^{op}\grc^{-1}(b) \cong \bigoplus_{b\in\mathfrak{B}} \grc^{-1}(b)\mathcal{H}_I(W).
	$$
	Hence, $\mathcal{H}(W)$ is a free right $\mathcal{H}_I(W)$-module of rank $r<\infty$.
		The converse is similar.
	
	We now specialise $\mathcal{H}(W) \cong \bigoplus_{b\in\mathfrak{B}}\mathcal{H}_I(W)b$ to $\mathbb{Z}[W]$ and we obtain  that $\mathbb{Z}[W] \cong \bigoplus_{b\in\mathfrak{B}}\mathbb{Z}[W_I] b$, whence $r=|W|/|W_I|$.
	
	\end{proof}

Thanks to Proposition \ref{lr}, there is an equivalent version of Conjecture \ref{conjp1}:
\begin{conj}[The one-sided parabolic freeness conjecture]\label{conjpp1}
	For each parabolic subgroup $W_I$ of $W$, the algebra $\mathcal{H}(W)$ is free as a left or  right $\mathcal{H}_I(W)$-module of finite rank.
\end{conj}
Conjecture \ref{conjpp1} holds for all real reflection groups \cite[\S 4.4.7, Relation (b)]{gepf}. For the infinite family $G(de,e,n)$, it has been proved for the parabolic subgroups of type $G(de,e,n-1)$ \cite[Theorem 2.4.21]{Sch}.

Our aim in the rest of the paper will be to prove the parabolic freeness conjecture for the exceptional groups of rank $2$ that we studied in \cite{BCCK} and \cite{BCC}. If $W$ is a complex reflection group of rank $2$, then all its non-trivial proper parabolic subgroups are of the form $W_H$ for some $H \in \mathcal{A}$.  As for the parabolic symmetrising trace conjecture, due to Proposition \ref{conjtogen}, it is enough to prove the parabolic freeness conjecture for every parabolic subgroup generated by a generator of $W$ (or at least a generator per conjugacy class).
Let now $s$ be a generator of $\mathcal{H}(W)$ and let $H$ be the  reflecting hyperplane of the image of $s$ in $W$. For $\mathcal{H}(W)$ to be a free left  
$\mathcal{H}_H(W)$-module of rank $r=|W|/|W_H|$, it is enough to find a subset $\mathfrak{B}_s^l(W)$ of $\mathcal{H}(W)$, with $|\mathfrak{B}_s^l(W)|=r$, such that the set
\begin{equation}\label{notationparab}
\mathfrak{P}_s^l(W):=\{s^jb\,|\, j=0,\ldots, |W_H|-1,\,b \in \mathfrak{B}_s^l(W)\}
\end{equation}
is a basis of $\mathcal{H}(W)$ as an $R(W)$-module. Then $\mathfrak{B}_s^l(W)$ is a basis of $\mathcal{H}(W)$ as a left $\mathcal{H}_H(W)$-module.
We call the set  $\mathfrak{P}_s^l(W)$ a \emph{left parabolic basis of $\mathcal{H}(W)$ with respect to the generator $s$}. We can similarly define $\mathfrak{B}_s^r(W)$ and $\mathfrak{P}_s^r(W)$ when considering $\mathcal{H}(W)$ as a right $\mathcal{H}_H(W)$-module.

As we explained in the last paragraph of Section \ref{gap3section}, the GAP3 program that we created allows us to check whether  any subset of the Hecke algebra $\mathcal{H}(W)$ consisting of $|W|$ elements is a basis of $\mathcal{H}(W)$ as an $R(W)$-module, for $W \in \{G_4,G_5,G_6,G_7,G_8,G_{13}\}$.  So, in the next section, we will explain how we came up with good candidates for left and right parabolic bases for these groups, thus proving directly the original parabolic freeness conjecture.

\subsection{Finding parabolic bases}\label{Finding parabolic bases}

From now on, for any group $G$, we will denote by $\overline{G}$ the quotient $G/Z(G)$. 
For every $g \in G$, we will denote by $\bar{g}$ its image under the natural surjection $G \rightarrow \overline{G}$.

Let $W$ be an exceptional irreducible complex reflection group of rank 2. Following the Shephard--Todd classification, this means that $W$ is one of the groups $G_4, G_5, \dots, G_{22}$. By \cite[Chapter 6]{lehrer}, we know that these groups divided into 3 smaller families, according to whether the group $\overline{W}$ is the tetrahedral group (which is the alternating group $\mathfrak{A}_4$), octahedral group (which is the symmetric  group $\mathfrak{S}_4$), or icosahedral group (which is the alternating group $\mathfrak{A}_5$). More precisely, we have the \emph{tetrahedral family}, which includes the groups $G_4,\dots, G_7$, the \emph{octahedral family}, which includes the groups $G_8,\dots, G_{15}$, and the \emph{icosahedral family}, which includes the rest.
In each family, there is a maximal group of order $|\overline{W}|^2$ and all the other groups are its subgroups. These are the groups $G_7$, $G_{11}$ and $G_{19}$. Moreover, the group $\overline{W}$ is isomorphic to the subgroup of a finite Coxeter group $C$ of rank 3 (of type $A_3$, $B_3$ and $H_3$ for the tetrahedral, octahedral and icosahedral family respectively), consisting of the elements of even Coxeter length.

We saw in Sections \ref{rc} and \ref{br}  that every complex reflection group has a Coxeter-like presentation and that its associated  braid group  has an Artin-like presentation.
We call this the \emph{BMR presentation}, due to Brou\'e, Malle and Rouquier.
In  \cite[\S 6.1]{ERrank2}, Etingof and Rains gave different presentations for the exceptional groups of rank 2 and their associated braid groups, based on the BMR presentations of the maximal groups $G_7$, $G_{11}$ and $G_{19}$. We call these the  \emph{ER presentations}.
In \cite[Appendix A]{Ch17},  the first author gives the two presentations for every $W$ and $B(W)$, as well as the isomorphisms between the BMR and ER presentations. Notice that for the maximal groups the ER presentations coincide with the BMR presentations. Moreover, the number of generators in the ER presentation is always $3$, while in the BMR presentation it can be either $2$ (for well-generated groups, such as $G_8$) or $3$ (for not well-generated groups, such as $G_7$ or $G_{13}$).

\subsubsection{Deformed Coxeter group algebras} Let $W$ be an exceptional group of rank 2 and let $(C,S)$ be a finite Coxeter system  of rank $3$ with $S=\{y_1,y_2, y_3\}$. We set $m_{ij}:=m_{y_iy_j}$ and $\tilde{\mathbb{Z}}:=\mathbb{Z} [{\rm exp}(2\pi i /m_{ij})]$. In \cite[\S 2]{ERcoxeter}, Etingof and Rains defined a $\tilde{\mathbb{Z}}$-algebra  $A(C)$ associated to $C$, with a  presentation given by:
\begin{itemize}[leftmargin=*]
	\item generators: $Y_1, Y_2, Y_3$, $t_{ij,k}$, where $i,j\in\{1,2,3\}$, $i\not=j$ and $k\in\mathbb{Z}/m_{ij}\mathbb{Z}$;
	\item relations: $Y_i^2=1$, $t_{ij,k}^{-1}=t_{ji,-k}$, $\prod\limits_{k=1}^{m_{ij}}(Y_iY_j-t_{ij,k})=0$, $t_{ij,k}Y_r=Y_rt_{ij,k}$, $t_{ij,k}t_{i'j',k'}=t_{i'j',k'}t_{ij,k}$.
\end{itemize}

Let $R^C:=\tilde{\mathbb{Z}}[t_{ij,k}^{\pm 1}]=\tilde{\mathbb{Z}}[t_{ij,k}]$.
The algebra $A(C)$ is naturally an $R^C$-algebra. The sub-$R^C$-algebra $A_+(C)$ generated by $Y_iY_j$,  for $i\not=j$, can be presented as follows:
\begin{itemize}[leftmargin=*]
	\item generators: $A_{ij}:=Y_iY_j$, where $i,j\in\{1,2,3\}$, $i\not=j$;
	\item relations: $A_{ij}^{-1}=A_{ji}$, $\prod\limits_{k=1}^{m_{ij}}(A_{ij}-t_{ij,k})=0$, $A_{ij}A_{jl}A_{li}=1$, for $\#\{i,j,l\}=3$.
\end{itemize}

The following result is \cite[Lemma 2.6]{Ch17} and gives a nice presentation of the algebra $A_+(C)$.
\begin{lem}Let $C$ be a finite Coxeter group of type  $A_3$, $B_3$ or $H_3$. We can present the $R^C$-algebra $A_+(C)$ as follows:$$\left\langle
	\begin{array}{l|cl}
	&(A_{13}-t_{13,1})( A_{13}-t_{13,2})=0&\\
	A_{13},  A_{32},  A_{21}&(A_{32}-t_{32,1})( A_{32}-t_{32,2})( A_{32}-t_{32,3})=0,&  A_{13} A_{32} A_{21}=1\\
	&( A_{21}-t_{21,1})( A_{21}-t_{21,2})\dots( A_{21}-t_{21,m})=0&
	\end{array}\right\rangle,
	$$
	where $m$ is $3$, $4$ or $5$ for each type respectively.
	\label{presentation}
\end{lem}

If $w$ is a word in letters $y_i$, we let $T_w$ denote the corresponding word in $Y_i$, an element of 
$A(C)$. For every $x\in \overline{W}$, let us choose a reduced word $w_x$ that represents $x$ in $C$. We notice that $T_{w_x}$ is an element in $A_+(C)$, since $w_x$ is reduced and $\overline{W}$ contains the elements of $C$ of even Coxeter length. 
\begin{exmp}{\rm
	Let $W$ be one of the exceptional groups belonging to the octahedral family, meaning that $\overline{W} \cong \mathfrak{S}_4$. As we mentioned before, $\overline{W}$ is isomorphic to the subgroup of  $B_3$ consisting of elements of even Coxeter length. For the reduced word  $w_x=y_1y_2y_1y_3$, we have that $T_{w_x}=A_{12}A_{13}=A_{21}^{-1}A_{13}$.}\end{exmp}
	
We have the following result \cite[Theorem 2.3(ii)]{ERcoxeter}:
\begin{thm}\label{defo}
	The algebra $A_+(C)$ is generated as an $R^C$-module by the elements $T_{w_x}$, $x\in \overline{W}$.
	\label{thmER}
\end{thm}

We will now see how the algebra $A_+(C)$ relates with the generic Hecke algebra $\mathcal{H}(W)$. Let $X$ be an indeterminate.
Following the notations of \cite[\textsection 2.2 ]{Mar43}, we set 
$R_{\tilde{\ZZ}}(W):=R(W)\otimes_{\ZZ}\tilde{\ZZ}$, $R_{\tilde{\ZZ}}^+(W):=R_{\tilde{\ZZ}}(W)[X,X^{-1}]$, and  $\mathcal{H}_{\tilde{\ZZ}}(W):=\mathcal{H}(W)\otimes_{R(W)}R_{\tilde{\ZZ}}(W)$. We have that the algebra $\mathcal{H}_{\tilde{\ZZ}}(W)$ inherits a structure of $R_{\tilde{\ZZ}}^+(W)$-module, with the action of $X$ described in \cite[Proposition 2.10]{Mar43} or \cite[Proposition 3.1]{Ch17}. Moreover, through this action, the algebra $\mathcal{H}_{\tilde{\ZZ}}(W)$ can be seen as a quotient of the group algebra $R_{\tilde{\ZZ}}^+(W)[\overline{B(W})]$. 
The following result \cite[Proposition 3.2]{Ch17} relates the algebra $A_+(C)$ with the algebra $\mathcal{H}_{\tilde \ZZ}(W)$. 
\begin{prop}\label{ERSUR}
	Let $W$ be an exceptional group of rank $2$, $W \notin \{G_{13},G_{15}\}$. There is a ring morphism $\gru: R^C\twoheadrightarrow R_{\tilde \ZZ}^+(W)$ inducing  a map $ \grC: A_+(C)\otimes_{\gru} R_{\tilde \ZZ}^+(W) \twoheadrightarrow \mathcal{H}_{\tilde{\ZZ}}(W)$ given by $A_{13}\mapsto \bar \gra$, $A_{32}\mapsto \bar \grb$, $A_{21}\mapsto \bar \grg$,  	where $\gra$, $\grb$, $\grg$ denote the generators of  $B(W)$ in the ER presentation.
\end{prop}

\begin{rem}\rm
If $W \in \{G_{13},G_{15}\}$, there exists a similar map $\Psi$, but for a version of $A_+(C)$ with specialised parameters  \cite[Proposition 3.3]{Ch17}.
\end{rem}

\subsubsection{Finding bases}\label{subbases}

For every exceptional group of rank 2 we call the surjection $\grC$ the \emph{ER surjection} associated to $W$.
In \cite{Ch17}, the first author used the ER surjection to construct bases for the generic Hecke algebras of the groups $G_4, \dots, G_{16}$, thus proving the  BMR freeness conjecture. We will use here similar techniques in order to construct  parabolic bases for the generic Hecke algebras of the groups $G_4$, $G_5$, $G_6$, $G_7$, $G_8$ and $G_{13}$, thus proving the parabolic freeness conjecture.

 Combining Theorem \ref{defo} and Proposition \ref{ERSUR}, we have that the elements $\grC({T}_{w_x})$, $x\in \overline{W}$,
 generate the
 algebra $\mathcal{H}_{\tilde{\ZZ}}(W)$ as an $R_{\tilde \ZZ}^+(W)$-module. 
 Motivated by this result, we construct some elements inside the braid group $B(W)$ and we prove that their images inside $\mathcal{H}(W)$ form a parabolic basis.
 We construct these braid elements as follows:
\begin{itemize}
\item We represent each element $x\in \overline{W}$ with a word $\tilde{w}_x$ (not necessarily reduced). We choose $\tilde{w}_x$ in such a way so that it can be considered as a word in letters $a_{ij}:=y_iy_j$ with $i \neq j$ (for example, we can take $\tilde{w}_x=y_1y_2y_2y_3=a_{12}a_{23}$, but not  $\tilde{w}_x=y_2y_2y_1y_3$). This is possible, since $\overline{W}$ is the subgroup of $C$ that contains the elements of even Coxeter length. We take $\tilde{1}=1$.

\item Using the fact that $a_{ij}^{-1}=a_{ji}$, we write $\tilde{w}_x$ as a word in letters $a_{13}$, $a_{32}$, $a_{21}$, and their inverses.

\item Inspired by the ER-surjection,  we define an element $\bar{b}_{\tilde{w}_x}$ inside $\overline{B(W)}$ by replacing $a_{13}$, $a_{32}$ and $a_{21}$  with $\bar{\gra}$, $\bar{\grb}$ and $\bar{\grg}$ respectively, where $\gra$, $\grb$, $\grg$ denote the generators of  $B(W)$ in the ER presentation.

\item We can use the group isomorphism $\grf_2$  described in  \cite[Appendix A, Table 2]{Ch17}  to write the elements $\gra$, $\grb$, $\grg$ in the BMR presentation. 
Hence, if we denote by $\Gamma(W)$ the set of generators of $B(W)$ in the BMR presentation, we can also  consider the element $\bar{b}_{\tilde{w}_x}$ as being
a product of elements $\bar{g}$, where $g \in \Gamma(W)$. We denote this last element by $\bar v_x$.

\item Let $\bar{g}_1^{m_1}\bar{g}_2^{m_2}\dots \bar{g}_r^{m_r}$ be the aforementioned factorisation of $\bar v_x$, with $g_i \in \Gamma(W)$  and $m_i\in \ZZ$. 
Let $f: \overline{B(W)}\rightarrow B(W)$ be a set theoretic section such that 
$f(\bar{g}_1^{m_1}\bar{g}_2^{m_2}\ldots \bar{g}_r^{m_r})=f(\bar{g}_1)^{m_1}f(\bar{g}_2)^{m_2}\ldots f(\bar{g}_r)^{m_r}$ and $f(\bar g)=g$, for every $g\in \Gamma(W)$. We set $v_x:=f(\bar v_x)$.
\end{itemize}

Recall now that $Z(B(W))$ is a cyclic group, generated by the element $\boldsymbol{\beta}$.
 Let $z$ denote the image of $\boldsymbol{\beta}$  inside $\mathcal{H}(W)$, while we keep the notation $v_x$ for the image of $v_x$ inside  $\mathcal{H}(W)$. 
 The set 
$\mathcal{C}:=\cup_{x\in \overline{W}} \mathcal{C}_{v_x}$, 
where $\mathcal{C}_{v_x}:=\{z^kv_x \,|\,k=0,1,\ldots,|Z(W)|-1\}$
 is a good candidate for a basis of $\mathcal{H}(W)$ as an $R(W)$-module.
 By construction, $1 \in \mathcal{C}$, so, in order to prove that the set $\mathcal{C}$ is indeed a basis, it is enough to prove that the set 
 $\sum_{h\in \mathcal{C}}R(W)h$ is an ideal of $\mathcal{H}(W)$. This is the approach used in \cite{Ch17}, and it managed to produce bases for all groups of the tetrahedral and the octahedral family.  It has  two main difficulties:
 \begin{itemize}
 \item[(1)] The element $\tilde{w}_x$ has to be chosen appropriately.\smallbreak
 \item[(2)] There are a lot of calculations involved  (see \cite[Appendix B]{Ch17}), which, at least for the moment, cannot be automated.
 \end{itemize}
 The choice of $\tilde{w}_x$  is a product of  experimentation, combined with experience. The second difficulty, when it comes to finding parabolic bases, can be solved for the groups that we studied in \cite{BCCK} and \cite{BCC} with the use of our GAP3 program and the method that we discussed in the last paragraph of  Section \ref{gap3section}.

Finally, there is a third difficulty that we encountered when trying to find a parabolic basis that is also good in the sense of \S \ref{good bases}:
 sometimes the set $\mathcal{C}_{v_x}$ would not work for any of the choices for $\tilde{w}_x$.
 This is a problem similar to the one that we encountered in \cite{BCCK}, when we wanted to find a good basis in order to prove the BMM symmetrising trace conjecture. The solution in both articles turns out to be the same. We have to break the pattern. 
Thus, for some $x \in  \overline{W}$, we  choose another word $\tilde w'_{x}\not=\tilde w_x$ that represents $x$ in $\overline{W}$.
Using the same procedure as before, we define an element $v'_x \in B(W)$. Then, in the definition of $\mathcal{C}$, 
we replace $\mathcal{C}_{v_x}$ by $\left(\mathcal{C}_{v_x} \setminus \{v_x\} \right)\cup \{v_x'\}$. Having a good parabolic basis can be useful for the proof of the parabolic trace conjecture: if $\mathcal{H}(W)$ admits a good parabolic basis with respect to a generator $s$, then $s$ satisfies Condition \eqref{crit}. Therefore, if $\mathcal{H}(W)$ admits a good parabolic basis with respect to every generator (or at least one generator per conjugacy class), then the parabolic trace conjecture holds.

 In the next sections, we will present the choices  of $\tilde{w}_x$ and $\tilde{w}'_x$  that we made in order to produce good  parabolic bases 
  for groups $G_4$, $G_5$, $G_6$, $G_7$, $G_8$ and $G_{13}$. In the case of $G_7$, the choice of these words was also made so that we can  easily obtain parabolic bases for $G_5$ and $G_6$. On the project's webpage \cite{web}, we give more parabolic bases for $G_7$ and $G_8$: some that have a pattern but are not good, and some where the pattern is broken simply so that all elements of the basis are positive words in the generators.

  \begin{rem}\rm
  The GAP3 program expresses every element $h \in \mathcal{C}$ as  a linear combination of elements of the good basis $\mathcal{B}(W)$ that we used to prove the BMM symmetrising trace conjecture in \cite{BCCK} and \cite{BCC}. The coefficient of $1$ in this linear combination is $\tau(h)$. Given the construction of $\mathcal{C}$, it is enough to verify that $\tau(h)=\delta_{1h}$ in order to show that $\mathcal{C}$ satisfies the lifting conjecture (Conjecture \ref{lif}). This amounts to checking that the first row of the change of basis matrix is $(1,0,0,\ldots,0)$ (assuming that $1$ is taken to be the first element of $\mathcal{C}$).
    \end{rem}

   From now on, we will use the notation of \eqref{notationparab}. Namely, for any generator $s$ of $\mathcal{H}(W)$, $\mathfrak{B}^l_s(W)$ (respectively $\mathfrak{B}^r_s(W)$) will denote a basis of $\mathcal{H}(W)$ as a left (respectively right) module over the corresponding parabolic Hecke subalgebra such that
   $$\mathfrak{P}_s^l(W)=\{s^jb\,|\, j=0,\ldots, |W_H|-1,\,b \in \mathfrak{B}_s^l(W)\} \quad \text{and} \quad \mathfrak{P}_s^r(W)=\{bs^j\,|\, j=0,\ldots, |W_H|-1,\,b \in \mathfrak{B}_s^r(W)\}$$
   are respectively left and right parabolic bases of $\mathcal{H}(W)$ with respect to $s$.

\subsubsection{The tetrahedral family} We recall that the tetrahedral family consists of the groups $G_4,\dots, G_7$, with $G_7$ being the maximal group in this family. 
The generic Hecke algebra of each group in this family can be seen as a subalgebra of $\mathcal{H}(G_7)$ for some specialisation of the parameters. We will use this fact in order
to construct parabolic bases for $G_5$ and $G_6$ from those for $G_7$. However, it is much simpler to deal with $G_4$ on its own to start with.
\\\\
\textbf{The case of} $\mathbf{G_4:}$ 
The generic Hecke algebra of $G_4$ is the $R(G_4)$-algebra
$$\mathcal{H}(G_4)=\left\langle s_1,\,s_2 \,\,\,\left|\,\,\, s_1s_2s_1=s_2s_1s_2,\,\,\, \prod_{i=1}^3(s_1-u_i) =  \prod_{i=1}^3(s_2-u_i) =0 \right\rangle\right.$$
where $R(G_4)=\Z[u_1^{\pm 1},u_2^{\pm 1},u_3^{\pm 1}]$
(note that $s_1$ and $s_2$ are conjugate). If we take
$$\mathfrak{B}_{s_1}^l(G_4)=\{1,\;z,\;s_2,\;s_2^2,\;s_2s_1,\;s_2^2s_1,\;s_2s_1^2,\;s_2^2s_1^2\} \quad \text{and} \quad \mathfrak{B}_{s_1}^r(G_4)=\{1,\;z,\;s_2,\;s_2^2,\;s_1s_2,\;s_1s_2^2,\;s_1^2s_2,\;s_1^2s_2^2\}$$
where $z=(s_1s_2)^3$, then $\mathfrak{P}_{s_1}^l(G_4)=\mathfrak{P}_{s_1}^r(G_4)=\mathcal{B}(G_4)$, the basis of \cite[\S 4.1.1]{BCCK} that we used for proving the BMM symmetrising trace conjecture. Given the symmetric role played by the generators $s_1$ and $s_2$ in the presentation of $\mathcal{H}(G_4)$, replacing $s_1$ with $s_2$ and $s_2$ with $s_1$ inside $\mathfrak{P}_{s_1}^l(G_4)$ and $\mathfrak{P}_{s_1}^r(G_4)$ yields $\mathfrak{P}_{s_2}^l(G_4)$ and $\mathfrak{P}_{s_2}^r(G_4)$ respectively.

Other parabolic bases for $\mathcal{H}(G_4)$ with respect to $s_1$ or $s_2$ can be excerpted from \cite[Corollary 3.3]{Mar41} (for example, one is explicitly given in \cite[Proposition 2.1]{Mar46}). Some of them satisfy the lifting conjecture, some of them do not (the one in  \cite[Proposition 2.1]{Mar46} does not). Since the generators of $G_4$ are conjugate, and due to Proposition \ref{conjtogen}, the existence of a parabolic basis with respect to one generator implies the validity of Conjecture \ref{conjpp1}, which is equivalent to the parabolic freeness conjecture.
\\\\
\textbf{The case of } $\mathbf{G_7:}$  The generic Hecke algebra of $G_7$ is the $R(G_7)$-algebra
$$\mathcal{H}(G_7)=\left\langle s_1,\,s_2,\,s_3 \,\,\,\left|\,\,\, s_1s_2s_3=s_2s_3s_1=s_3s_1s_2,\,\,\, \prod_{i=1}^2(s_1-u_{s_1,i}) =  \prod_{j=1}^3(s_2-u_{s_2,j}) =\prod_{k=1}^3(s_3-u_{s_3,k})=0 \right\rangle\right.$$
where $R(G_7)=\Z[u_{s_1,1}^{\pm 1},u_{s_1,2}^{\pm 1}, u_{s_2,1}^{\pm 1},u_{s_2,2}^{\pm 1},u_{s_2,3}^{\pm 1},u_{s_3,1}^{\pm 1},u_{s_3,2}^{\pm 1},u_{s_3,3}^{\pm 1}]$.
If we take
$$\mathfrak{B}_{s_2}^r(G_7)=\left(\{z^k,\; z^ks_3,\; z^ks_3^2,\;z^ks_2s_3^{-1} \,|\,k=0,1,\ldots,11\} \setminus \{ s_2s_3^{-1}\}\right) \cup \{s_2s_3^2\},$$
where $z=s_1s_2s_3$,   then $\mathfrak{P}_{s_2}^r(G_7)=\mathcal{B}(G_7)$, the basis of \cite[\S 4.2.4]{BCCK} that we used for proving the BMM symmetrising trace conjecture. 
Even without replacing $s_2s_3^{-1}$ by $s_2s_3^2$, we obtain a right parabolic basis with respect to $s_2$, which does not satisfy though the lifting conjecture for $\tau$ (this is in fact the basis for $\mathcal{H}(G_7)$ constructed in \cite{Ch17} with the procedure described in \S\ref{subbases}).

We will now use the procedure of \S\ref{subbases} to construct the remaining parabolic bases for $\mathcal{H}(G_7)$, and we will explicitly describe the steps for the right parabolic bases. For $G_7$, we recall that the ER presentation coincides with the BMR presentation, so we have 
$\grC(A_{13})=\bar{s_1}$, $\grC(A_{32})=\bar{s_2}$, $\grC(A_{21})=\bar{s_3}$. 
Moreover, $\overline{G_7}$ is isomorphic to the subgroup of elements of even Coxeter length of the group
$$A_3=\langle y_1,y_2,y_3 \,|\, y_1^2=y_2^2=y_3^2=(y_1y_3)^2=(y_3y_2)^3=(y_2y_1)^3=1\rangle \cong \mathfrak{S}_4.$$
The elements of $\overline{G_7}$ can be represented with the following reduced words in letters $y_1, y_2, y_3$:
\begin{equation}\label{reducedG7}
1,\;y_1y_2,\; y_2y_1,\; y_2y_3,\; y_3y_2,\; y_1y_3,\; y_2y_1y_3y_2,\; y_2y_3y_2y_1,\; y_1y_3y_2y_1, \;
  y_1y_2y_1y_3, \; y_1y_2y_3y_2,\; y_1y_2y_1y_3y_2y_1.
  \end{equation}

First, we will apply the procedure of \S\ref{subbases} to construct a right parabolic basis with respect to $s_1$.
 Thus, for every element $x\in\overline{G_7}$, we will try to find a word  $\tilde w_x$ in letters $y_1,y_2,y_3$ that represents $x$ in $\overline{G_7}$, so that 
 \begin{itemize}
 \item the set $P_1:=\{x\in\overline{G_7}\,|\,\tilde w_x \text{ does not end in } y_1y_3\}$ has $6$ elements;\smallbreak
 \item the set $P_2:=\{x\in\overline{G_7}\,|\,\tilde w_x=\tilde w_{x'}y_1y_3 \text{ for some } x' \in P_1\}$ has $6$ elements.
 \end{itemize}
 We start with the elements that have reduced expressions ending in $y_1y_3$. These belong to $P_2$ and they are:
 \begin{itemize}
 \item $y_1y_3$: We take $\widetilde{y_1y_3}=y_1y_3=a_{13}$. \smallbreak
 \item $y_1y_2y_1y_3$: In order to avoid $a_{21}^{-1}$ appearing, we take $\widetilde{y_1y_2y_1y_3}=y_2y_1y_2y_1y_1y_3=a_{21}^2a_{13}.$\smallbreak
 \item $y_2y_3y_2y_1$: We have $y_2y_3y_2y_1=y_3y_2y_3y_1=y_3y_2y_1y_3$. We take $\widetilde{y_3y_2y_1y_3}=y_3y_2y_1y_3=a_{32}a_{13}$.\smallbreak
 \item $y_1y_2y_1y_3y_2y_1$: We have $y_1y_2y_1y_3y_2y_1=y_2y_1y_2y_3y_2y_1=y_2y_1y_3y_2y_3y_1=y_2y_1y_3y_2y_1y_3$. We take
 $\widetilde{y_2y_1y_3y_2y_1y_3}=y_2y_1y_3y_2y_1y_3=a_{21}a_{32}a_{13}$.
 \end{itemize}
 We multiply the above elements with $y_3y_1$ from the right (using the reduced expression ending in $y_1y_3$ yields reduced expressions) and obtain the following elements of $P_1$:
  \begin{itemize}
 \item $1$: We take $\widetilde{1}=1$. \smallbreak
 \item $y_1y_2$: We take $\widetilde{y_1y_2}=y_2y_1y_2y_1=a_{21}^2$.\smallbreak
 \item $y_3y_2$: We take $\widetilde{y_3y_2}=y_3y_2=a_{32}$.\smallbreak
 \item $y_2y_1y_3y_2$: We take $\widetilde{y_2y_1y_3y_2}=y_2y_1y_3y_2=a_{21}a_{32}$. 
 \end{itemize}
 It remains to deal with the last 4 elements:
 \begin{itemize}
 \item $y_2y_1$: We take $\widetilde{y_2y_1}=y_2y_1=a_{21}$. \smallbreak
 \item $y_2y_3$: We take $\widetilde{y_2y_3}=y_2y_1y_1y_3=a_{21}a_{13}$. \smallbreak
 \item $y_1y_3y_2y_1$: We have $y_1y_3y_2y_1=y_3y_1y_2y_1=y_3y_2y_1y_2$. We take $\widetilde{y_3y_2y_1y_2}=y_3y_2y_1y_2=a_{32}a_{21}^{-1}$. \smallbreak
 \item $y_1y_2y_3y_2$: We have $y_1y_2y_3y_2=y_1y_3y_2y_3=y_3y_1y_2y_3$. We take $\widetilde{y_3y_1y_2y_3}=y_3y_2y_1y_2y_1y_3=a_{32}a_{21}^{-1}a_{13}$.
 \end{itemize}
 We note that in the last two elements, we cannot avoid $a_{21}^{-1}$ if we want to obtain a basis.
 
	We now construct the elements $\bar v_x$ by replacing $a_{13}$, $a_{32}$ and $a_{21}$  with $\bar{s_1}$, $\bar{s_2}$ and $\bar{s_3}$ respectively. 
	We have the following set, consisting of the elements $\bar v_x$:
	$$\{1,\;\bar{s_1},\; \bar{s_2}, \;\bar{s_2}\bar{s_1},\;\bar{s_3}, \;\bar{s_3}\bar{s_1},\; \bar{s_3}^2,\; \bar{s_3}^2\bar{s_1},\; \bar{s_3}\bar{s_2},\; \bar{s_3}\bar{s_2}\bar{s_1},\; \bar{s_2}\bar{s_3}^{-1},\; \bar{s_2}\bar{s_3}^{-1}\bar{s_1}\}.$$
	The set of the elements $v_x$ is now the following:
	$$\{1,\; s_1, \;s_2, \;s_2s_1,\;s_3,\;s_3s_1,\;s_3^2,\; s_3^2s_1,\; s_3s_2,\; s_3s_2s_1,\; s_2s_3^{-1},\; s_2s_3^{-1}s_1\}.$$
	If we take
        $$\mathfrak{B}_{s_1}^r(G_7)=\{z^k, \; z^ks_2, \; z^ks_3,\; z^ks_3^2,\;z^ks_3s_2,\;z^ks_2s_3^{-1} \,|\,k=0,1,\ldots,11\},$$
	 then the GAP3 program yields that the corresponding $\mathfrak{P}^r_{s_1}(G_7)$ is a good right parabolic basis with respect to $s_1$.

We will now construct a right parabolic basis with respect to $s_3$.
For every element $x\in\overline{G_7}$, we will try to find a word  $\tilde w_x$ in letters $y_1,y_2,y_3$ that represents $x$ in $\overline{G_7}$, so that 
 \begin{itemize}
 \item the set $P_1:=\{x\in\overline{G_7}\,|\,\tilde w_x \text{ does not end in } y_2y_1\}$ has $4$ elements;\smallbreak
 \item the set $P_2:=\{x\in\overline{G_7}\,|\,\tilde w_x=\tilde w_{x'}y_2y_1 \text{ for some } x' \in P_1\}$ has $4$ elements;\smallbreak
 \item the set $P_3:=\{x\in\overline{G_7}\,|\,\tilde w_x=\tilde w_{x'}(y_2y_1)^2 \text{ for some } x' \in P_1\}$ has $4$ elements.
 \end{itemize}
 Since $(y_2y_1)^2=y_1y_2$, we start with the elements that have reduced expressions ending in $y_1y_2$. These belong to $P_3$ and they are:
 \begin{itemize}
 \item $y_1y_2$: We take $\widetilde{y_1y_2}=y_2y_1y_2y_1=a_{21}^2$. \smallbreak
 \item $y_1y_3y_2y_1$: We have  $y_1y_3y_2y_1=y_3y_1y_2y_1=y_3y_2y_1y_2$. We take $\widetilde{y_3y_2y_1y_2}=y_3y_2y_2y_1y_2y_1=a_{32}a_{21}^2$. \smallbreak
 \item $y_2y_1y_3y_2$: We have $y_2y_1y_3y_2 = y_2y_3y_1y_2$. In order to avoid $a_{32}^{-1}$ appearing, we take $\widetilde{y_2y_3y_1y_2}=y_3y_2y_3y_2y_2y_1y_2y_1=a_{32}^2a_{21}^2$.\smallbreak
 \item $y_1y_2y_1y_3y_2y_1$: We have $y_1y_2y_1y_3y_2y_1=y_1y_2y_3y_1y_2y_1=y_1y_2y_3y_2y_1y_2$. We take $\widetilde{y_1y_2y_3y_2y_1y_2}=
 y_1y_2y_3y_2y_2y_1y_2y_1=a_{21}^{-1}a_{32}a_{21}^2$.
 \end{itemize}
  We multiply the above elements with $y_1y_2$ from the right and obtain the following elements of $P_2$:
  \begin{itemize}
 \item $y_2y_1$: We take $\widetilde{y_2y_1}=y_2y_1=a_{21}$. \smallbreak
 \item $y_1y_3$: We take $\widetilde{y_1y_3}=y_3y_2y_2y_1=a_{32}a_{21}$.\smallbreak
 \item $y_2y_3y_2y_1$: We take $\widetilde{y_2y_3y_2y_1}=y_3y_2y_3y_2y_2y_1=a_{32}^2a_{21}$.\smallbreak
 \item $y_1y_2y_1y_3$: We take $\widetilde{y_1y_2y_1y_3}=y_1y_2y_3y_2y_2y_1=a_{21}^{-1}a_{32}a_{21}$. 
 \end{itemize}
 Finally, we multiply the first $4$ elements with $y_1y_2y_1y_2=y_2y_1$ from the right (using the reduced expression ending in $y_1y_2$ yields reduced expressions) and obtain the following elements of $P_1$:
 \begin{itemize}
 \item $1$: We take $\widetilde{1}=1$. \smallbreak
 \item $y_3y_2$: We take $\widetilde{y_3y_2}=y_3y_2=a_{32}$. \smallbreak
 \item $y_2y_3$: We take $\widetilde{y_2y_3}=y_3y_2y_3y_2=a_{32}^2$. \smallbreak
 \item $y_1y_2y_3y_2$:  We take $\widetilde{y_1y_2y_3y_2}=y_1y_2y_3y_2=a_{21}^{-1}a_{32}$.
 \end{itemize}

We now construct the elements $v_x$ as before, by replacing $a_{13}$, $a_{32}$ and $a_{21}$  with $s_1$, $s_2$ and $s_3$ respectively (we skip the construction of the intermediate $\bar v_x$, which is implied). We have the following set, consisting of the elements $v_x$:
$$\{1,\; s_3, \;s_3^2, \;s_2, \;s_2s_3,\; s_2s_3^2, \;s_2^2,\;s_2^2s_3,\; s_2^2s_3^2,\; s_3^{-1}s_2,\; s_3^{-1}s_2s_3,\; s_3^{-1}s_2s_3^2\}.$$
If we take
$$\mathfrak{B}_{s_3}^r(G_7)=\{z^k,\; z^ks_2,\; z^ks_2^2,\;z^ks_3^{-1}s_2 \,|\,k=0,1,\ldots,11\},$$
 then the GAP3 program yields that the corresponding $\mathfrak{P}^r_{s_3}(G_7)$ is a good right parabolic basis with respect to $s_3$.

Using now similar techniques, we have found good left parabolic bases $\mathfrak{P}_{s_i}^l(G_7)$ for $i=1,2,3$. These correspond to:
$$\begin{array}{lcl}
\mathfrak{B}^l_{s_1}(G_7)&=&\{z^k,\;z^ks_2,\;z^ks_3,\;z^ks_2^2,\; z^ks_3s_2,\;z^ks_2^{-1}s_3\,|\,
\; k=0,\dots, 11\} ,\smallbreak\smallbreak\smallbreak\\
\mathfrak{B}^l_{s_2}(G_7)&=&\left(\{z^k,\; z^ks_3,\; z^ks_3^2, \; z^ks_3^{-1}s_2\,|\,k=0,\dots, 11\} \setminus \{s_3^{-1}s_2\}\right) \cup \{s_3^2s_2\},\smallbreak\smallbreak\smallbreak\\
\mathfrak{B}^l_{s_3}(G_7)&=&\{z^k,\; z^ks_2,\; z^ks_2^2, \;z^ks_2s_3^{-1}\,|\,k=0,\dots, 11\}.\end{array}$$

We now observe that the bases
 $\mathfrak{P}^r_{s_1}(G_7)$ and $\mathfrak{P}^l_{s_1}(G_7)$ are both examples of a good basis that contains the subset
$$\{z^k, \; z^{l}s_1,\; z^ks_2,\; z^ks_3 \,|\,k=0,1,\ldots,11,\,\,l=0,1,\ldots,5\}.$$
The existence of such a basis was assumed by Malle for the determination of the Schur elements with respect to the canonical symmetrising trace  for all groups of the tetrahedral family  in \cite{Ma2}.  We have thus obtained the following:

\begin{thm}
The Schur elements with respect to the canonical symmetrising trace  for groups $G_4$, $G_5$, $G_6$ and $G_7$ are the ones computed in  \cite[\S 4B]{Ma2}.
\end{thm}

Finally, we will use the left and right parabolic bases for $\mathcal{H}(G_7)$ described above in order to construct left and right parabolic bases for $\mathcal{H}(G_5)$ and $\mathcal{H}(G_6)$. We first need to prove the following lemma about the element $z$:

\begin{lem}\label{center}
	For the central element $z=s_1s_2s_3$ of $\mathcal{H}(G_7)$ and for every  $k\in \NN$ we have:
	\begin{itemize}
		\item[(a)]$z^k=s_1^k(s_2s_3)^k=(s_2s_3)^ks_1^k$.\smallbreak
		\item[(b)]$z^k=s_2^k(s_3s_1)^k=(s_3s_1)^ks_2^k$.\smallbreak
		\item[(c)]$z^k=s_3^k(s_1s_2)^k=(s_1s_2)^ks_3^k$.
	\end{itemize}
\end{lem}
\begin{proof}
	We prove the three properties by induction to $k$. For $k=1$, these are given by the braid relation $s_1s_2s_3=s_2s_3s_1=s_3s_1s_2$. 
	We now assume, for (a), that $z^{k-1}=s_1^{k-1}(s_2s_3)^{k-1}$. We have $z^k=z^{k-1}s_1s_2s_3=s_1z^{k-1}s_2s_3= s_1s_1^{k-1}(s_2s_3)^{k-1}s_2s_3=s_1^k(s_2s_3)^k$. Since $s_1$ commutes with $s_2s_3$, we also have $z^k=(s_2s_3)^ks_1^k$. Similarly we prove (b) and (c).
\end{proof}

\noindent
\textbf{The case of} $\mathbf{G_6:}$ 
 Let $\theta: R(G_7)\mapsto \Z[u_{s_1,0}^{\pm 1},u_{s_1,1}^{\pm 1}, u_{s_3,0}^{\pm 1},u_{s_3,1}^{\pm 1},u_{s_3,2}^{\pm 1}]\cong R(G_6)$ be a specialisation, defined by 
$$(u_{s_1,1},u_{s_1,2};\; u_{s_2,1},u_{s_2,2},u_{s_2,3};\; u_{s_3,1},u_{s_3,2},u_{s_3,3})\mapsto(u_{s_1,1},u_{s_1,2};\; 1,\zeta_3,\zeta_3^2;\;u_{s_3,1},u_{s_3,2},u_{s_3,1}),$$
where $\zeta_3$ denotes a primitive cubic root of unity.
Let $A:=\mathcal{H}(G_7)\otimes_{\theta} R(G_6)$, that is,
$$A=\left\langle s_1,\,s_2,\,s_3 \,\,\,\left|\,\,\, s_1s_2s_3=s_2s_3s_1=s_3s_1s_2,\,\,\,s_2^3=1,\; \prod_{i=1}^2(s_1-u_{s_1,i}) =  \prod_{k=1}^3(s_3-u_{s_3,k})=0 \right\rangle\right.,$$ 
and let $\bar A$  be the subalgebra of $A$ generated by $s_1$ and $s_3$. Since the BMR freeness conjecture holds for all exceptional groups of the tetrahedral family by \cite{Ch17}, we have that $\bar{A}$ is isomorphic to $\mathcal{H}(G_6)$  \cite[Proposition 4.2 \& Table 4.6]{Ma2} and that
\begin{equation}A=\bigoplus\limits_{j=0}^2s_2^j\bar A=\bigoplus\limits_{j=0}^2\bar As_2^j
\label{g6}
\end{equation}
(see \cite[\S 4]{Ma2} or \cite[Appendix A.1]{chlou}).

Let us now explain how we obtain a right parabolic basis $\mathfrak{P}_{s_1}^r(G_6)$ from $\mathfrak{P}_{s_1}^r(G_7)$. 
 Let $A_1$ denote the subalgebra of $A$ generated by $s_1$.
The algebra $A$ is generated as a right $A_1$-module by  the set
$$\mathfrak{B}_{s_1}^r(G_7)=\{z^k, \; z^ks_2,\; z^ks_3,\; z^ks_3^2,\;z^ks_3s_2,\;z^ks_2s_3^{-1} \,|\,k=0,1,\ldots,11\},$$
where $z=s_1s_2s_3$.  Using the fact that $s_2^3=1$, Lemma \ref{center} yields:
$$z^k=\begin{cases}
(s_3s_1)^k, &\text { for } k=0,3,6,9,\\
(s_3s_1)^ks_2, &\text { for } k=1,4,7,10,\\
(s_3s_1)^ks_2^2, &\text { for } k=2,5,8,11.\\
\end{cases}$$
Since the sum in \eqref{g6} is direct, the subalgebra $\bar A$ of $A$ is generated as a right $A_1$-module by the set
$$\mathfrak{B}:=\{(s_3s_1)^k, \;(s_3s_1)^ks_3,\;(s_3s_1)^ks_3^2 \,|\, k=0,3,6,9\}\cup \{(s_3s_1)^k,\;s_3(s_3s_1)^k,\;(s_3s_1)^ks_3^{-1} \,|\, k=2,5,8,11\}.
$$
As a consequence, the set $\{bs_1^i\,|\, i=0,1,\,b \in \mathfrak{B}\}$ generates $\bar A$  as an $R(G_6)$-module.
Now, the rank of $\mathcal{H}(G_6) \cong \bar A$ as a free $R(G_6)$-module is equal to $|G_6|=48=2|\mathfrak{B}|$. This implies that the  set $\mathfrak{B}$ is linearly independent, and it is the set $\mathfrak{B}^r_{s_1}(G_6)$ that we are looking for. Moreover, since $\mathfrak{P}^r_{s_1}(G_6) \subseteq \mathfrak{P}^r_{s_1}(G_7)$, and the canonical symmetrising trace on $\mathcal{H}(G_6)$ is the restriction of the canonical symmetrising trace on $\mathcal{H}(G_7)$ \cite[Lemme 4.3]{Ma2}, the right parabolic basis $\mathfrak{P}^r_{s_1}(G_6)$ is also a good basis.

Similarly, we obtain good parabolic bases by taking:
$$\begin{array}{lcl}
\mathfrak{B}^l_{s_1}(G_6)&=&
\{(s_3s_1)^k,\;(s_3s_1)^ks_3\,|\,k=0,1,3,4,6,7,9,10\} \cup
\{(s_3s_1)^k,\;  s_3(s_3s_1)^k\,|\,k=2,5,8,11\},
\smallbreak\smallbreak\smallbreak\\
\mathfrak{B}^r_{s_3}(G_6)&=&
\{(s_3s_1)^k \,|\,k=0,\dots,11\} \cup  \{s_3^{-1}(s_3s_1)^k \,|\,k=2,5,8,11\},\smallbreak\smallbreak\smallbreak\\
\mathfrak{B}^l_{s_3}(G_6)&=&
\{(s_3s_1)^k \,|\,k=0,\dots,11\}\cup 
\{(s_3s_1)^ks_3^{-1} \,|\,k=2,5,8,11\}.
\end{array}
$$\medskip

\noindent
\textbf{The case of} $\mathbf{G_5:}$ 
Let $\theta: R(G_7)\mapsto \Z[u_{s_2,1}^{\pm 1},u_{s_2,2}^{\pm 1},u_{s_2,3}^{\pm 1}, u_{s_3,1}^{\pm 1},u_{s_3,2}^{\pm 1},u_{s_3,3}^{\pm 1}]\cong R(G_5)$ be a specialisation, defined by 
$$(u_{s_1,1},u_{s_1,2};\; u_{s_2,1},u_{s_2,2},u_{s_2,3};\; u_{s_3,1},u_{s_3,2},u_{s_3,3})\mapsto(1,-1;\; u_{s_2,1},u_{s_2,2},u_{s_2,3};\;u_{s_3,1},u_{s_3,2},u_{s_3,3}).$$
Let $A:=\mathcal{H}(G_7)\otimes_{\theta} R(G_5)$, that is,
$$A=\left\langle s_1,\,s_2,\,s_3 \,\,\,\left|\,\,\, s_1s_2s_3=s_2s_3s_1=s_3s_1s_2,\,\,\,s_1^2=1,\; \prod_{j=1}^3(s_2-u_{s_2,j}) =  \prod_{k=1}^3(s_3-u_{s_3,k})=0 \right\rangle\right.,$$ and let $\bar A$  be the subalgebra of $A$ generated by $s_2$ and $s_3$. 
Since the BMR freeness conjecture holds for all exceptional groups of the tetrahedral family by \cite{Ch17}, we have that $\bar{A}$ is isomorphic to $\mathcal{H}(G_5)$  \cite[Proposition 4.2 \& Table 4.6]{Ma2} and that
\begin{equation}A=\bigoplus\limits_{i=0}^1s_1^i\bar A=\bigoplus\limits_{i=0}^1\bar As_1^i
\label{g5}
\end{equation}
(see \cite[\S 4]{Ma2} or \cite[Appendix A.1]{chlou}).

As in the case of $G_6$, we will construct parabolic bases for $\mathcal{H}(G_5)$ from the ones we constructed for $\mathcal{H}(G_7)$. We will use the example of
the right parabolic basis with respect to $s_2$ to illustrate our method.

 Let $A_2$ denote the subalgebra of $A$ generated by $s_2$.
The algebra $A$ is generated as a right $A_2$-module by  the set
$$\mathfrak{B}_{s_2}^r(G_7)=\left(\{z^k,\; z^ks_3,\; z^ks_3^2,\;z^ks_2s_3^{-1} \,|\,k=0,1,\ldots,11\} \setminus \{ s_2s_3^{-1}\}\right) \cup \{s_2s_3^2\},$$
where $z=s_1s_2s_3$.  Using the fact that $s_1^2=1$, Lemma \ref{center} yields:
\begin{equation}\label{gg5}
z^k=\begin{cases}
(s_2s_3)^k,&\text{ for } k=0,2,4,6,8,10,\\
s_1(s_2s_3)^k,&\text{ for } k=1,3,5,7,9,11.
\end{cases}
\end{equation}
Since the sum in \eqref{g5} is direct, the subalgebra $\bar A$ of $A$ is generated as a right $A_2$-module by the set
$$\mathfrak{B}:=\left(\{(s_2s_3)^k,\; (s_2s_3)^ks_3,\; (s_2s_3)^ks_3^2,\;(s_2s_3)^ks_2s_3^{-1} \,|\,k=0,2,4,6,8,10\} \setminus \{ s_2s_3^{-1}\}\right) \cup \{s_2s_3^2\}.$$
As a consequence, the set $\{bs_2^j\,|\, j=0,1,2,\,b \in \mathfrak{B}\}$ generates $\bar A$  as an $R(G_5)$-module.
Now, the rank of $\mathcal{H}(G_5) \cong \bar A$ as a free $R(G_5)$-module is equal to $|G_5|=72=3|\mathfrak{B}|$. This implies that the  set $\mathfrak{B}$ is linearly independent, and it is the set $\mathfrak{B}^r_{s_2}(G_5)$ that we are looking for. Moreover, since $\mathfrak{P}^r_{s_2}(G_5) \subseteq \mathfrak{P}^r_{s_2}(G_7)$, and the canonical symmetrising trace on $\mathcal{H}(G_5)$ is the restriction of the canonical symmetrising trace on $\mathcal{H}(G_7)$ \cite[Lemme 4.3]{Ma2}, the right parabolic basis $\mathfrak{P}^r_{s_2}(G_5)$ is also a good basis.

Similarly, we obtain good parabolic bases by taking:
$$\begin{array}{lcl}\mathfrak{B}^l_{s_2}(G_5)&=&\left(\{(s_2s_3)^k,\; (s_2s_3)^ks_3,\; (s_2s_3)^ks_3^2,\; (s_2s_3)^ks_3^{-1}s_2\,|\,k=0,2,4,6,8,10\} \setminus \{s_3^{-1}s_2\}\right) \cup \{s_3^2s_2\},\smallbreak\smallbreak\smallbreak\\
\mathfrak{B}_{s_3}^r(G_5)&=&\{(s_2s_3)^k,\; (s_2s_3)^ks_2,\; (s_2s_3)^ks_2^2,\;(s_2s_3)^ks_3^{-1}s_2 \,|\,k=0,2,4,6,8,10\},\smallbreak\smallbreak\smallbreak\\
\mathfrak{B}^l_{s_3}(G_5)&=&\{(s_2s_3)^k,\; (s_2s_3)^ks_2,\; (s_2s_3)^ks_2^2,\; (s_2s_3)^ks_2s_3^{-1}\,|\,k=0,2,4,6,8,10\}.
\end{array}
$$
\medskip

\subsubsection{The octahedral family}\label{813}
The octahedral family consists of the groups $G_8,\dots,G_{15}$, with the maximal group being $G_{11}$. Since we are not able yet to deal with group $G_{11}$, we cannot apply here the same method that we used for $G_5$ and $G_6$, so we will deal with $G_8$ and $G_{13}$ on their own.
For both groups though, as for all the groups of the octahedral family, the associated group $\overline{G_n}$ ($n=8,\ldots,15$) is isomorphic to the subgroup of elements of even Coxeter length of the group
$$B_3=\langle y_1,y_2,y_3 \,|\, y_1^2=y_2^2=y_3^2=(y_1y_3)^2=(y_3y_2)^3=(y_2y_1)^4=1\rangle \cong (\Z/2\Z)^3 \rtimes \mathfrak{S}_3.$$
The elements of $\overline{G_n}$ can be represented with the following reduced words in letters $y_1, y_2, y_3$:
\begin{equation}\label{reducedG11}
\begin{array}{c}
1, \; y_3 y_2 , \; y_2 y_1 , \; y_2 y_3 , \; y_1 y_3 , \; y_1 y_2 , \; y_3 y_2 y_1 y_2 , \;
   y_2 y_1 y_2 y_3 , \; y_2 y_3 y_2 y_1 , \; y_2 y_1 y_3 y_2 , \; y_2 y_1 y_2 y_1 , \;
   y_1 y_3 y_2 y_1,\\  y_1 y_2 y_1 y_3 , \; y_1 y_2 y_3 y_2 ,  \;
   y_2 y_3 y_2 y_1 y_2 y_3 , \;
   y_2 y_1 y_2 y_3 y_2 y_1 ,\;  y_2 y_1 y_3 y_2 y_1 y_2 ,\;  y_1 y_3 y_2 y_1 y_2 y_3 , \;
   y_1 y_2 y_1 y_2 y_3 y_2 ,\\  y_1 y_2 y_1 y_3 y_2 y_1 , \; y_1 y_2 y_3 y_2 y_1 y_2 , \;
   y_2 y_1 y_2 y_3 y_2 y_1 y_2 y_3 , \; y_1 y_2 y_1 y_2 y_3 y_2 y_1 y_2 , \;
   y_1 y_2 y_1 y_3 y_2 y_1 y_2 y_3. 
   \end{array}
  \end{equation}
  
  \noindent
\textbf{The case of} $\mathbf{G_8:}$  The generic Hecke algebra of $G_8$ is the $R(G_8)$-algebra
$$\mathcal{H}(G_8)=\left\langle s_1,\,s_2 \,\,\,\left|\,\,\, s_1s_2s_1=s_2s_1s_2,\,\,\, \prod_{i=1}^4(s_1-u_i) =  \prod_{i=1}^4(s_2-u_i) =0 \right\rangle\right..$$
where $R(G_8)=\Z[u_1^{\pm 1},u_2^{\pm 1},u_3^{\pm 1},u_4^{\pm 1}]$ (note that $s_1$ and $s_2$ are conjugate).
If we take
$$\mathfrak{B}_{s_1}^l(G_8)=\{z^k, \;z^ks_2, \;z^ks_2^2,\;z^ks_2s_1,\; z^ks_2s_1^2,\; z^ks_2s_1^3 \,|\,k=0,1,2,3\} ,$$
where $z=(s_1s_2)^3$,   then $\mathfrak{P}_{s_1}^l(G_8)=\mathcal{B}(G_8)$, the basis of \cite[\S 4.2.1]{BCCK} that we used for proving the BMM symmetrising trace conjecture.

We will now use the procedure of \S\ref{subbases} to construct  $\mathfrak{P}_{s_1}^r(G_8)$. For $G_8$, the isomorphism $\grf_2$ between the ER and BMR presentation is given by:
\begin{center}
 $\grf_2(\alpha)=(s_1s_2s_1)^{-1}$, $\grf_2(\beta)=s_1s_2$, and
 $\grf_2(\gamma)=s_1$.
 \end{center}
Hence, $s_1=\grf_2(\gamma)$ and $s_2=\grf_2(\gamma^{-1}\beta)$. We recall that $\grC(A_{21})=\bar \gamma$. Since we look for a right parabolic basis with respect to $s_1$, 
for every element $x\in \overline{G_8}$, we will try to find a word $\tilde w_x$ in letters $y_1, y_2$ and $y_3$ that represents $x$ in $\overline{G_8}$, so that
\begin{itemize}
 \item the set $P_1:=\{x\in\overline{G_8}\,|\,\tilde w_x \text{ does not end in } y_2y_1\}$ has $6$ elements;\smallbreak
 \item the set $P_2:=\{x\in\overline{G_8}\,|\,\tilde w_x=\tilde w_{x'}y_2y_1 \text{ for some } x' \in P_1\}$ has $6$ elements;\smallbreak
 \item the set $P_3:=\{x\in\overline{G_8}\,|\,\tilde w_x=\tilde w_{x'}(y_2y_1)^2 \text{ for some } x' \in P_1\}$ has $6$ elements;\smallbreak
 \item the set $P_4:=\{x\in\overline{G_8}\,|\,\tilde w_x=\tilde w_{x'}(y_2y_1)^3  \text{ for some } x' \in P_1\}$ has $6$ elements.\smallbreak
 \end{itemize}
 Since $(y_2y_1)^3=y_1y_2$, we start with some  elements that have reduced expressions ending in $y_1y_2$ and belong to $P_4$. These are:
\begin{itemize}
\item $y_1y_2$: We take $\widetilde{y_1y_2}=y_2y_1y_2y_1y_2y_1=a_{21}^3$.\smallbreak
\item $y_3y_2y_1y_2$: We take $\widetilde{y_3y_2y_1y_2}=y_3y_2y_2y_1y_2y_1y_2y_1=a_{32}a_{21}^3$.\smallbreak
\item  $y_1y_2y_3y_2y_1y_2$:  We take $\widetilde{y_1y_2y_3y_2y_1y_2}=y_1y_2y_3y_2y_2y_1y_2y_1y_2y_1=a_{21}^{-1}a_{32}a_{21}^3$.\smallbreak
\end{itemize} 
 We multiply the above elements with $y_1y_2$ from the right and obtain the following elements of $P_3$:
  \begin{itemize}
 \item $y_2y_1y_2y_1$: We take $\widetilde{y_2y_1y_2y_1}=y_2y_1y_2y_1=a_{21}^2$. \smallbreak
 \item $y_1y_3y_2y_1$: We take $\widetilde{y_1y_3y_2y_1}=y_3y_2y_2y_1y_2y_1=a_{32}a_{21}^2$.\smallbreak
\item  $y_1y_2y_1y_3y_2y_1$:  We take $\widetilde{y_1y_2y_1y_3y_2y_1}=y_1y_2y_3y_2y_2y_1y_2y_1=a_{21}^{-1}a_{32}a_{21}^2$.\smallbreak
 \end{itemize}
 We then multiply the first $3$ elements with $(y_1y_2)^2=(y_2y_1)^2$ from the right and obtain the following elements of $P_2$:
  \begin{itemize}
 \item $y_2y_1$: We take $\widetilde{y_2y_1}=y_2y_1=a_{21}$. \smallbreak
 \item $y_1y_3$: We take $\widetilde{y_1y_3}=y_3y_2y_2y_1=a_{32}a_{21}$.\smallbreak
\item  $y_1y_2y_1y_3$:  We take $\widetilde{y_1y_2y_1y_3}=y_1y_2y_3y_2y_2y_1=a_{21}^{-1}a_{32}a_{21}$.\smallbreak
 \end{itemize}
 Finally, we multiply the first $3$ elements with $(y_1y_2)^3=y_2y_1$ from the right (using the reduced expression ending in $y_1y_2$ yields reduced expressions) and obtain the following elements of $P_1$:
 \begin{itemize}
 \item $1$: We take $\widetilde{1}=1$. \smallbreak
 \item $y_3y_2$: We take $\widetilde{y_3y_2}=y_3y_2=a_{32}$. \smallbreak
 \item  $y_1y_2y_3y_2$:  We take $\widetilde{y_1y_2y_3y_2}=y_1y_2y_3y_2=a_{21}^{-1}a_{32}$.\smallbreak
 \end{itemize}
For the remaining elements we choose the following words $\tilde w_x$:
 \begin{itemize}
\item $(y_1y_2y_3y_2)^2$, \;$(y_1y_2y_3y_2)^2(y_2y_1)$, \;$(y_1y_2y_3y_2)^2(y_2y_1)^2$, \; $(y_1y_2y_3y_2)^2(y_2y_1)^3$. \smallbreak
\item  $(y_1y_2y_3y_2)^3$, \; $(y_1y_2y_3y_2)^3(y_2y_1)$, \;$(y_1y_2y_3y_2)^3(y_2y_1)^2$, \; $(y_1y_2y_3y_2)^3(y_2y_1)^3$.\smallbreak
\item $y_2y_1y_2y_3y_2y_1$, \; $y_2y_1y_2y_3y_2y_1(y_2y_1)$, \; $y_2y_1y_2y_3y_2y_1(y_2y_1)^2$, \;$y_2y_1y_2y_3y_2y_1(y_2y_1)^3$.\smallbreak
 \end{itemize}

 We now construct the elements $v_x$ as before, by replacing in every word $\tilde w_x$ the products $y_1y_3$, $y_3y_2$ and $y_2y_1$ with $(s_1s_2s_1)^{-1}$, $s_1s_2$ and $s_1$ respectively (using also inverses, if necessary). We have the following set, consisting of the elements $v_x$:
$$\left\{
\begin{matrix}
\begin{array}{cl}
1,\;s_1,\; s_1^2, \;s_1^3,\; s_2,\; s_2s_1,\; s_2s_1^2,\; s_2s_1^3,\; s_2^2,\; s_2^2s_1,\; s_2^2s_1^2,\; s_2^2s_1^3,\;
s_2^3, \;s_2^3s_1,\; s_2^3s_1^2,\; s_2^3s_1^3,\;\;&\\   \;s_1s_2,\; s_1s_2s_1,\; s_1s_2s_1^2,\; s_1s_2s_1^3, \;s_1s_2^{-1},\;s_1s_2^{-1}s_1,\; s_1s_2^{-1}s_1^2,\; s_1s_2^{-1}s_1^3
\end{array}
\end{matrix}\right\}$$
If we take
$\tilde{\mathfrak{B}}_{s_1}^r(G_8)=\{z^k, z^ks_2, z^ks_2^2,z^ks_2^3,z^ks_1s_2,z^ks_1s_2^{-1} \,|\,k=0,1,2,3\}$,
then the GAP3 program yields that the corresponding $\tilde{\mathfrak{P}}^r_{s_1}(G_8)$ is a right parabolic basis with respect to $s_1$, which does not satisfy though the lifting conjecture. 
More precisely, $\tau(s_1s_2^{-1}s_1^3)\not=0$. Therefore, if we want our basis to also satisfy the lifting conjecture, we need to replace the element $s_1s_2^{-1}s_1^3$. Since we want to preserve the  parabolic structure of the basis, we need to replace the elements $s_1s_2^{-1}$, $s_1s_2^{-1}s_1$ and $s_1s_2^{-1}s_1^2$ as well. All the aforementioned elements come from the elements $x \in \overline{G_8}$  
  for which we have chosen the words $\tilde{w}_x$ as follows: $y_2y_1y_2y_3y_2y_1(y_2y_1)^m$ (and, hence, $v_x=s_1s_2^{-1}s_1^{-1}s_1s_1^m=s_1s_2^{-1}s_1^m)$, for $m=0,1,2,3$. We notice that 
  $y_2y_1y_2y_3y_2y_1=y_3y_2y_1y_2y_3y_2y_1y_2y_3y_2$. Therefore, we choose the following words $\tilde w'_{x}$: $y_3y_2y_1y_2y_3y_2y_1y_2y_3y_2(y_1y_2)^m$.
  The corresponding elements $v'_x$ are: $s_1s_2s_1^{-1}s_1s_2s_1^{-1}s_1s_2s_1^m=s_1s_2^3s_1^m$, for $m=0,1,2,3$.
If now we take
$$\mathfrak{B}_{s_1}^r(G_8)=\left(\{z^k,\; z^ks_2,\; z^ks_2^2,\;z^ks_2^3,\;z^ks_1s_2,\;z^ks_1s_2^{-1} \,|\,k=0,1,2,3\} \setminus \{s_1s_2^{-1}\}\right) \cup \{s_1s_2^3\},$$
then the corresponding $\mathfrak{P}^r_{s_1}(G_8)$ is a good right parabolic basis with respect to $s_1$.

Using now similar techniques, we  have found good  parabolic bases with respect to $s_2$. These correspond to:
$$\begin{array}{lcl}\mathfrak{B}^l_{s_2}(G_8)&=&\{z^k,\;z^ks_1,\;z^ks_1^2,\;z^ks_1^3,\;z^ks_1s_2,\;z^ks_1s_2^{-1}\,|\,
\; k=0,1,2,3\},\smallbreak\smallbreak\smallbreak\\
\mathfrak{B}^r_{s_2}(G_8)&=&\left(\{z^k,\;z^ks_1,\;z^ks_1^2,\;z^ks_1^3,\;z^ks_2s_1,\;z^ks_2s_1^{-1}\,|\,k=0,1,2,3\} \setminus \{s_2s_1^{-1}\}\right) \cup \{s_2s_1^3\}.\end{array}$$
 One can notice that $\mathfrak{B}^r_{s_2}(G_8)$ can be obtained from $\mathfrak{B}^r_{s_1}(G_8)$ by  replacing everywhere $s_1$ with $s_2$ and $s_2$ with $s_1$. In fact, given the symmetric role played by the generators $s_1$ and $s_2$ in the presentation of $\mathcal{H}(G_8)$, replacing $s_1$ with $s_2$ and $s_2$ with $s_1$ inside a good parabolic basis with respect to one generator yields a good parabolic basis with respect to the other. \bigskip

\noindent
\textbf{The case of} $\mathbf{G_{13}:}$
The generic Hecke algebra of $G_{13}$ is the $R(G_{13})$-algebra
$$\mathcal{H}(G_{13}) =\left \langle\, {s_1}, {s_2}, {s_3}\,\left|\,\begin{array}{c}
 {s_2} {s_3} {s_1}{s_2}={s_3}{s_1} {s_2}{s_3},\,{s_1} {s_2}{s_3}{s_1}{s_2}=
{s_3}{s_1} {s_2}{s_3}{s_1},\,\\ \smallbreak
  \prod_{i=1}^2(s_1-u_{{s_1},i})=\prod_{j=1}^2(s_2-u_{{s_2},j})=
\prod_{j=1}^2(s_3-u_{{s_2},j})=0
\end{array}
\,\right\rangle\right..$$
where $R(G_{13})=\Z[u_{s_1,1}^{\pm 1},u_{s_1,2}^{\pm 1},u_{s_2,1}^{\pm 1},u_{s_2,2}^{\pm 1}]$ (notice that here $s_2$ and $s_3$ are conjugate).
If we take
$$\mathfrak{B}_{s_2}^r(G_{13}) =
\left\{
\begin{matrix}
\begin{array}{c|l}
z^k,\; z^ks_3,\; z^ks_1,\; z^ks_2s_1,\; z^ks_1s_3,\; z^ks_3s_1,\;
z^ks_2s_3, 
\\ \smallbreak
z^ks_2s_1s_3,\;z^ks_2s_3s_1,\;
z^ks_1s_2s_1,\;z^ks_1s_2s_3,\;z^ks_3s_2s_1
\end{array}
\end{matrix}\,\,\,k=0,1,2,3 \right\},
$$
where $z=(s_1s_2s_3)^3$,   then $\mathfrak{P}_{s_2}^r(G_{13})=\mathcal{B}(G_{13})$, the basis of \cite{BCC} that we used for proving the BMM symmetrising trace conjecture.

We now wish to use the procedure of \S\ref{subbases} to construct  the rest of the parabolic bases. For $G_{13}$, the isomorphism $\grf_2$ between the ER and BMR presentation is given by:
\begin{center}
 $\grf_2(\alpha)=s_2$, $\grf_2(\beta)=s_3s_1s_2$, and
 $\grf_2(\gamma)=(s_2s_3s_1s_2)^{-1}$.
 \end{center}
 Hence, $s_1=\grf_2(\gamma\beta^2\alpha^{-1})$, $s_2=\grf_2(\alpha)$ and $s_3=\grf_2(\beta^{-1}\gamma^{-1})$. We observe that this time describing a generator as the image of an element in the ER presentation is more complicated than in the previous cases. For this reason, we don't have a systematic way of choosing the words $\tilde w_x$ for $x \in \overline{G_{13}}$. The general idea is the same as before. Let us say, for example, that we try  to find a right parabolic basis with respect to $s_1$. Since
 $\grC(A_{21}A_{21})=\grC(A_{21}A_{32}A_{32}A_{31})=\bar{\gamma}\bar{\beta}^2\bar{\alpha}^{-1}$, 
 we try to find  words $\tilde w_x$ that represent $x$ in $\overline{G_{13}}$ so that
\begin{itemize}
 \item the set $P_1:=\{x\in\overline{G_{13}}\,|\,\tilde w_x \text{ does not end in } (y_2y_1)^2 \}$ has $12$ elements;\smallbreak
 \item the set $P_2:=\{x\in\overline{G_{13}}\,|\,\tilde w_x=\tilde w_{x'}(y_2y_1)^2 \text{ for some } x' \in P_1\}$ has $12$ elements.
  \end{itemize}
Thanks to the GAP3 program, we were able to try several different choices of words for some elements, until we finally found that the following sets produce good parabolic bases:

$$\begin{array}{lcl}\mathfrak{B}_{s_1}^r(G_{13}) &=&
\left\{
\begin{matrix}
\begin{array}{c|l}
z^k, \;  z^ks_2,\;z^ks_3,\;z^ks_2s_3,\;z^ks_1s_2,\;z^ks_3s_2,\;z^ks_1s_3,\\
z^ks_2s_1s_3,\;z^ks_1s_3s_2,\;z^ks_2s_3s_2,\; z^ks_2s_1s_2,\,z^ks_2s_1s_3s_2
\end{array}
\end{matrix}\,\,\,k=0,1,2,3 \right\},\smallbreak\smallbreak\smallbreak
\\

\mathfrak{B}_{s_3}^r(G_{13}) &=&
\left\{
\begin{matrix}
\begin{array}{c|l}
z^k, \;  z^ks_1,\;z^ks_2,\;z^ks_1s_2,\;z^ks_3s_1,\;z^ks_2s_1,\;z^ks_3s_2,\\
z^ks_2s_1s_2,\;z^ks_1s_2s_1,\; z^ks_3s_2s_1,\;z^ks_3s_1s_2,\;z^ks_3s_1s_2s_1
\end{array}
\end{matrix}\,\,\,k=0,1,2,3 \right\},
\smallbreak\smallbreak\smallbreak\\

\mathfrak{B}_{s_1}^l(G_{13}) &=&
\left\{
\begin{matrix}
\begin{array}{c|l}
z^k, \;  z^ks_3,\;z^ks_2,\;z^ks_2s_1,\;z^ks_3s_1,\;z^ks_2s_3,\;z^ks_3s_2,\\
z^ks_2s_3s_2,\;z^ks_2s_1s_2,\;
z^ks_2s_3s_1,\;
z^ks_3s_2s_1,\;z^ks_2s_3s_1s_2
\end{array}
\end{matrix}\,\,\,k=0,1,2,3 \right\},
\smallbreak\smallbreak\smallbreak\\

\mathfrak{B}_{s_2}^l(G_{13}) &=&
\left\{
\begin{matrix}
\begin{array}{c|l}
z^k, \;  z^ks_3,\;z^ks_1,\;z^ks_1s_2,\;z^ks_3s_1,\;z^ks_3s_2,\;z^ks_1s_3,\;\\
z^ks_1s_2s_1,\;z^ks_1s_2s_3,\; z^ks_3s_1s_2,\;
z^ks_1s_3s_2,\;z^ks_3s_2s_1 \phantom{aa}
\end{array}
\end{matrix}\,\,\,k=0,1,2,3 \right\},
\smallbreak\smallbreak\smallbreak\\

\mathfrak{B}_{s_3}^l(G_{13}) &=&
\left\{
\begin{matrix}
\begin{array}{c|l}
z^k, \;  z^ks_1,\;z^ks_2,\;z^ks_2s_1,\;z^ks_1s_2,\;z^ks_1s_3,\; z^ks_2s_3,\\
\;z^ks_2s_1s_2,\;z^ks_2s_1s_3,\;
z^ks_1s_2s_3,\;
z^ks_1s_2s_1,\;z^ks_2s_3s_1 \phantom{a,}
\end{array}
\end{matrix}\,\,\,k=0,1,2,3 \right\}.
\end{array}$$

\begin{rem}\rm
For both groups of the octahedral family that we studied here, $G_8$ and $G_{13}$, it is true that two generators of the generic Hecke algebra are conjugate, so we could have used the conjugacy relation to obtain a parabolic basis with respect to one generator from the one with respect to its conjugate (this is also true for $G_4$ that we studied much earlier). However, the bases that we obtain like this are much more complicated than the ones we presented here.
\end{rem}

\begin{rem}\rm
If in the future we obtain parabolic bases for $\mathcal{H}(G_{11})$, we will be able to use them in order to construct parabolic bases for the generic Hecke algebras of the groups of the octahedral family, in the same way that we used the ones of $\mathcal{H}(G_{7})$ in order to construct parabolic bases for $\mathcal{H}(G_5)$ and $\mathcal{H}(G_6)$.
\end{rem}

\subsubsection{The other exceptional groups}
Let $W$ be an exceptional complex reflection group of rank $2$ and let $\mathcal{H}(W)$ be the generic Hecke algebra associated with $W$. It is possible that a basis for $\mathcal{H}(W)$ that was constructed for proving the BMR freeness conjecture is a parabolic basis with respect to one generator -- we already saw such examples in the cases of $G_4$, $G_7$, $G_8$ and $G_{13}$. Moreover, in order for the parabolic freeness conjecture to hold for $W$, it is enough to 
\begin{itemize}
\item[(1)] prove the ``one-sided'' parabolic freeness conjecture (Conjecture \ref{conjpp1}), and \smallbreak
\item[(2)] find a parabolic basis with respect to only one generator per conjugacy class in $W$ (Proposition \ref{conjtogen}).
\end{itemize}
For example, one parabolic basis is all that is needed to prove the validity of the parabolic freeness conjecture if all generators of $W$ belong to the same conjugacy class. Among the groups that we studied, this is the case for $G_4$ and $G_8$.
So we took a look at the bases that exist in literature, and thanks to the above arguments, we also have the validity of the parabolic freeness conjecture for the following groups:
\begin{itemize}
\item $G_{12}$ (the bases given in \cite{Ch17} and in \cite{MaPf} are parabolic);\smallbreak
\item $G_{14}$ (the basis given in \cite{Ch17} is parabolic with respect to both generators);\smallbreak
\item $G_{16}$ (the  basis given in \cite{Ch18} is parabolic);\smallbreak
\item $G_{22}$ (the  basis given in  \cite{MaPf}  is parabolic).
\end{itemize}

Moreover, looking again at the bases given in \cite{Ch17}, we observe that Conjecture \ref{conjpp1} holds
for the following pairs $(W,W_I)$ (that is, $\mathcal{H}(W)$ is free either as a left or as a right $\mathcal{H}_I(W)$-module of finite rank):
		\begin{center}
$(G_9, \mathbb{Z}/4\mathbb{Z})$, $(G_{10},\mathbb{Z}/4\mathbb{Z})$, $(G_{11},\mathbb{Z}/3\mathbb{Z})$, $(G_{13}, \mathbb{Z}/2\mathbb{Z})$, $(G_{15}, \mathbb{Z}/2\mathbb{Z})$, $(G_{15}, \mathbb{Z}/3\mathbb{Z})$.
\end{center}

Finally, for the state of the art of the parabolic freeness conjecture to be complete we give a list of pairs   $(W,W_I)$ for which Conjecture \ref{conjpp1} holds, where $W$ is an exceptional group of rank greater than $2$:
	\begin{itemize} 
	\item $(G_{24}, B_2)$ by \cite{MaPf};\smallbreak
	\item $(G_{25},G_4)$, $(G_{25},\mathbb{Z}/3\mathbb{Z}\times \mathbb{Z}/3\mathbb{Z})$ by \cite{Mar41};\smallbreak
	\item $(G_{26},{G_4})$ by \cite{Mar43};\smallbreak
	\item $(G_{27},B_2)$, $(G_{29},B_3)$, $(G_{31}, A_3)$ by \cite{MaPf};\smallbreak
	\item $(G_{32},G_{25})$ by \cite{Mar41};\smallbreak
	\item $(G_{33}, A_4)$, $(G_{33},D_4)$, $(G_{34},G_{33})$ by \cite{MaPf}.
\end{itemize}

\end{document}